\documentclass[reqno,11pt]{amsart}
\usepackage[colorlinks=true, pdfstartview=FitV, linkcolor=blue, 
            citecolor=blue, urlcolor=blue]{hyperref}

\usepackage{graphicx}
\usepackage[mathscr]{euscript} 
\usepackage{bm}

\usepackage{amssymb}
\usepackage{amsmath}
\usepackage{amsthm}
\usepackage{pgf}
\usepackage{color}
\usepackage{graphicx}
\usepackage{comment}
\usepackage{mathtools}
\usepackage{hyperref}

\newcommand{\FFF}{\color{black}}

\newcommand{\PPP}{\color{black}}

\newcommand{\RRR}{\color{black}}
\newcommand{\BBB}{\color{black}}
\newcommand{\EEE}{\color{black}}
\newcommand{\UUU}{\color{black}}
\newcommand{\MMM}{\color{black}}
\newcommand{\GGG}{\color{black}}
\newcommand{\OOO}{\color{black}}
\newcommand{\ZZZ}{\color{black}}

\newcommand{\tone}{\boldsymbol t_1}
\newcommand{\ttwo}{\boldsymbol t_2}
\newcommand{\tthree}{\boldsymbol t_3}
\newcommand{\ti}{\boldsymbol t_i}

\newtheorem{theorem}{Theorem}[section]
\newtheorem{lemma}[theorem]{Lemma}
\newtheorem{proposition}[theorem]{Proposition}
\newtheorem{corollary}[theorem]{Corollary}

\theoremstyle{definition}
\begingroup
\newtheorem{definition}[theorem]{Definition}
\newtheorem{remark}[theorem]{Remark}

\endgroup

\newcommand{\Qz}{\mathbb{Q}}
\newcommand{\Rz}{\mathbb{R}}
\newcommand{\Nz}{\mathbb{N}}
\newcommand{\Zz}{\mathbb{Z}}

\newcommand{\measurerestr}{%
	\,\raisebox{-.127ex}{\reflectbox{\rotatebox[origin=br]{-90}{$\lnot$}}}\,%
}

\newcommand{\be}[1]{\begin{equation}\label{#1}}
\newcommand{\ee}{\end{equation}}

\begin{document}

\title[Microscopical justification of Solid-State Wetting and \FFF Dewetting \EEE]{Microscopical justification of Solid-State\\Wetting and Dewetting}

\author{Paolo Piovano}
\address[Paolo Piovano]{Department of Mathematics, University of Vienna,   Oskar-Morgenstern-Platz 1, A-1090 Vienna, Austria}
\email{paolo.piovano@univie.ac.at}

\author{Igor Vel\v{c}i\'c}
\address[Igor Vel\v{c}i\'c]{Faculty of Electrical Engineering and Computing, University of Zagreb, Unska 3, 10000 Zagreb, Croatia}
\email{igor.velcic@fer.hr}


\begin{abstract}
The continuum model related to the \FFF \emph{Winterbottom problem}, \EEE i.e., the problem of determining the equilibrium shape of crystalline drops resting on a substrate, is derived in dimension two by means of a rigorous discrete-to-continuum  passage by  $\Gamma$-convergence of atomistic models taking into consideration the atomic interactions of the drop particles both among themselves and with the fixed substrate atoms. As a byproduct of the analysis effective expressions for the drop surface anisotropy and the drop/substrate adhesion parameter  appearing in the  continuum model are characterized  in terms of the atomistic potentials, which are chosen of Heitmann-Radin sticky-disc type. Furthermore, a threshold condition only depending on such potentials is determined distinguishing the wetting regime, where discrete minimizers are \FFF explicitly \EEE characterized as configurations contained in \FFF a layer with a one-atom thickness, i.e., the wetting layer, \EEE  on the substrate,  from the dewetting regime. In the latter regime,  also in view of a proven conservation of mass in the limit as the number of atoms tends to infinity, proper scalings of \FFF the minimizers of the atomistic models \EEE converge (up to extracting a subsequence and performing translations on the substrate surface)  to a bounded minimizer of the Winterbottom continuum model \FFF satisfying the volume constraint. \EEE
 \end{abstract}

\maketitle

\pagestyle{myheadings}

\section{Introduction}

\PPP
The problem of determining the equilibrium shape formed by \FFF crystalline drops \EEE resting upon a rigid substrate possibly of a different material is  \FFF long-standing in materials science and applied mathematics. 
\FFF The first phenomenological prediction of such shape  for flat substrates is due to W.\ L.\ Winterbottom, who in  \cite{Winterbottom} designed what  is now referred to as  the \emph{Winterbottom construction}  (see Figure \ref{fig: winterbottom}) to minimize the  drop surface energy in which both 
 the \emph{drop anisotropy} at the free surface and the \emph{drop wettability} at the contact region with the substrate were taken into account (see \eqref{winterbottomsurfaceenergy} below).
  The interplay between the drop material properties of anisotropy and wettability 
  can induce different morphologies ranging from the spreading of the drops in a infinitely thick \emph{wetting layer} covering the substrate, which is exploited, e.g.,   in the design of  film coatings, to the nucleation of  \emph{dewetted islands}, that are solid-state clusters of atoms leaving the substrate exposed among them, which find other applications, such as for sensor devices and as catalysts for the growth of carbon and semiconductor nanowires  \cite{srolovitz1, srolovitz2}. 

\FFF In this work we introduce a discrete setting dependent on the atomistic interactions of drop particles both among themselves and with the substrate particles,  
and we characterize  in terms of the parameters of the  potentials governing such atomistic interactions the regime associated with the wetting layer, referred in the following as the \emph{wetting regime}.  
For the complementary parameter range, i.e., the \emph{dewetting regime}, we microscopically  justify the formation of solid-state dewetted  islands  
 by performing a rigorous discrete-to-continuum passage by means of showing the $\Gamma$-convergence of the atomistic energies to the  energy considered in   \cite{srolovitz1,srolovitz2} and by  W.\ L.\ Winterbottom in  \cite{Winterbottom}.   \EEE



In the continuum setting,  the \emph{Winterbottom problem}   in  \cite{Winterbottom} essentially consists in an optimization problem based on an \emph{a priori} knowledge of the \FFF surface anisotropy \EEE $\Gamma$ of the  resting crystalline drop with the surrounding vapor, and of the \emph{adhesivity} $\sigma$ related to the contact interface between the drop and the substrate. In the modern mathematical formulation in $\Rz^d$ for $d>1$ the energy associated to \FFF an admissible \EEE region $D\subset\Rz^d\setminus S$  occupied by \FFF the drop \EEE material, \FFF which is assumed to be a set of finite perimeter \EEE 
outside a fixed smooth substrate region $S\subset\Rz^d$, is given by 
\begin{equation}\label{winterbottomsurfaceenergy} 
 \mathcal{E}(D):=\int_{\partial^* D\setminus\partial S}\Gamma(\nu(\xi))\,  \mathrm{d}\mathcal{H}^{d-1}(\xi) +  \sigma \mathcal{H}^{1}(\partial^* D\cap\partial S),
\end{equation}
where 
 $\partial^* D$ is the reduced boundary of $D$,  $\nu$ is the exterior normal vector of $D$, and $\mathcal{H}^{d-1}$ the $(d-1)$-dimensional measure. The Winterbottom shape $W_{\Gamma,\sigma}$ introduced in  \cite{Winterbottom} is defined as depicted in Figure \ref{fig: winterbottom} by 
$$W_{\Gamma,\sigma}:=W_\Gamma\cap\{x\in\Rz^d\,:\, x_d\geq - \sigma\}$$
\PPP where $W_\Gamma$ is the \emph{Wulff shape}, i.e.,
$$
W_\Gamma:=\{x\in\Rz^d\,:\, x\cdot \nu\leq \Gamma(\nu) \text{ for every } \nu\in S^{d-1}\}. 
$$ 
The Wulff shape $W_\Gamma$ is named after G. Wulff, who provided in  \cite{Wulff01} its first phenomenological construction as  the equilibrium shape for a \FFF free-standing crystal with anisotropy $\Gamma$ in the space   (in the absence of a substrate or any other crystalline materials), and was afterwards in \cite{F,FM2}  rigorously proved to be the unique minimum of \eqref{winterbottomsurfaceenergy}  when  $S=\emptyset$ in the presence of a volume constraint  and after a proper scaling to adjust its volume  (see also  \cite{T,T2}).  
\EEE

The emergence of the Wulff and Winterbottom shapes have been already \FFF justified \EEE starting from discrete models in the context of \FFF statistical mechanics \EEE and the Ising model.   We refer to the review  \cite{DKS} (see also \cite{ISc,KP}) for  the 2-dimensional derivation of the Wulff shape in the scaling limit at low-temperature and  to \cite{Dodineau-etal,PV1,PV2} for the setting related to the Winterbottom shape. More recently, the microscopical justification of the Wulff shape in the context of atomistic models depending on Heitmann-Radin sticky-disc type potentials \cite{Heitmann-Radin80} has been addressed for $d=2$ and the triangular lattice  in \cite{Yeung-et-al12} by performing  a rigorous discrete-to-continuum analysis by means of $\Gamma$-convergence. 
Subsequetly,  the deviation of discrete ground states in the triangular lattice from the asymptotic Wulff shape has been sharply quantified in \cite{Schmidt} by introducing the $n^{3/4}$ law (see also \cite{DPS1}), which  has been then extended to  the square lattice in \cite{MPS,MPS2}, to the hexagonal lattice for graphene nanoflakes  in   \cite{DPS2}, and to higher dimensions in  \cite{MPSS,MS}.

\begin{figure}
\includegraphics[width=0.49\textwidth]{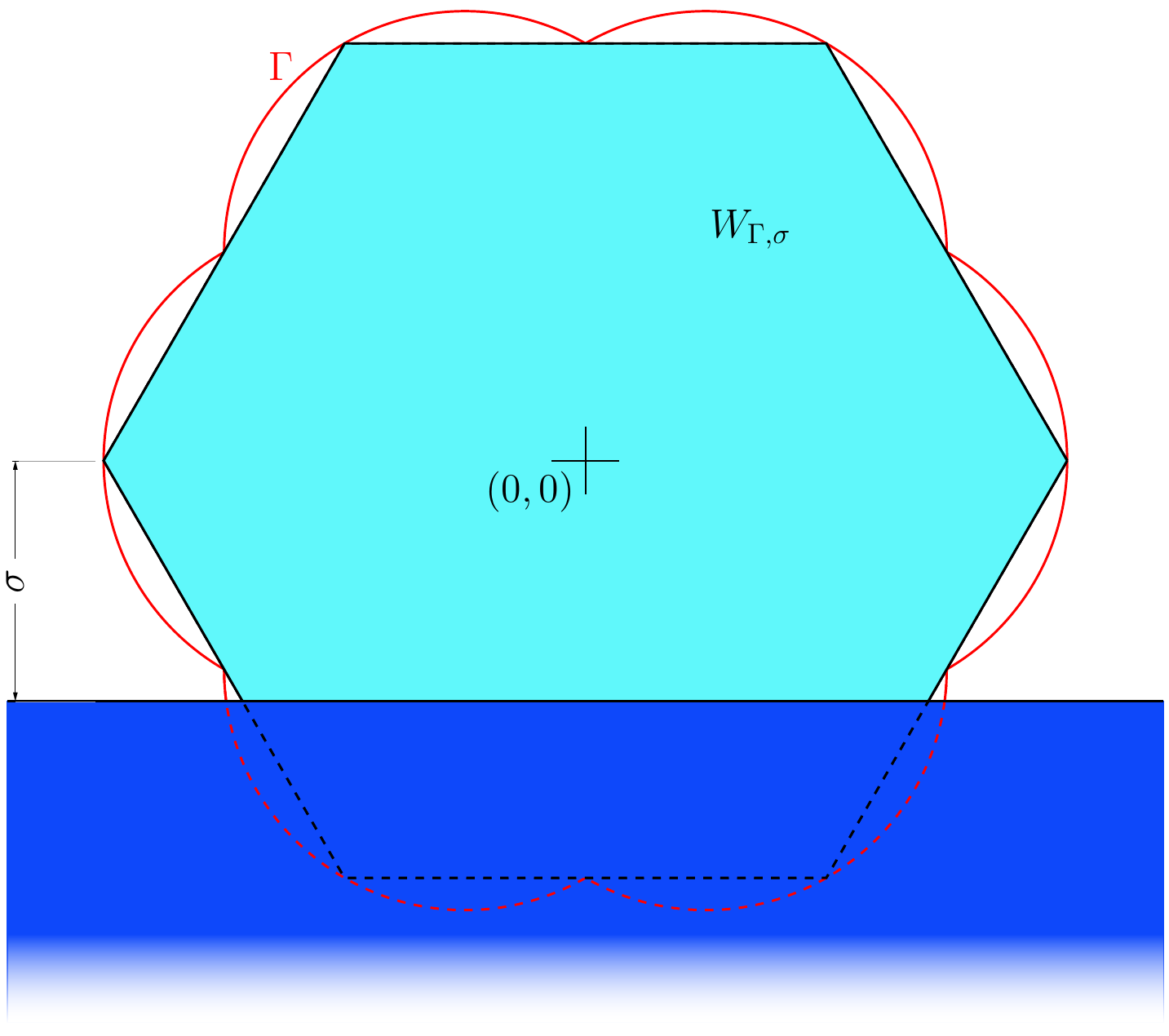}
\includegraphics[width=0.49\textwidth]{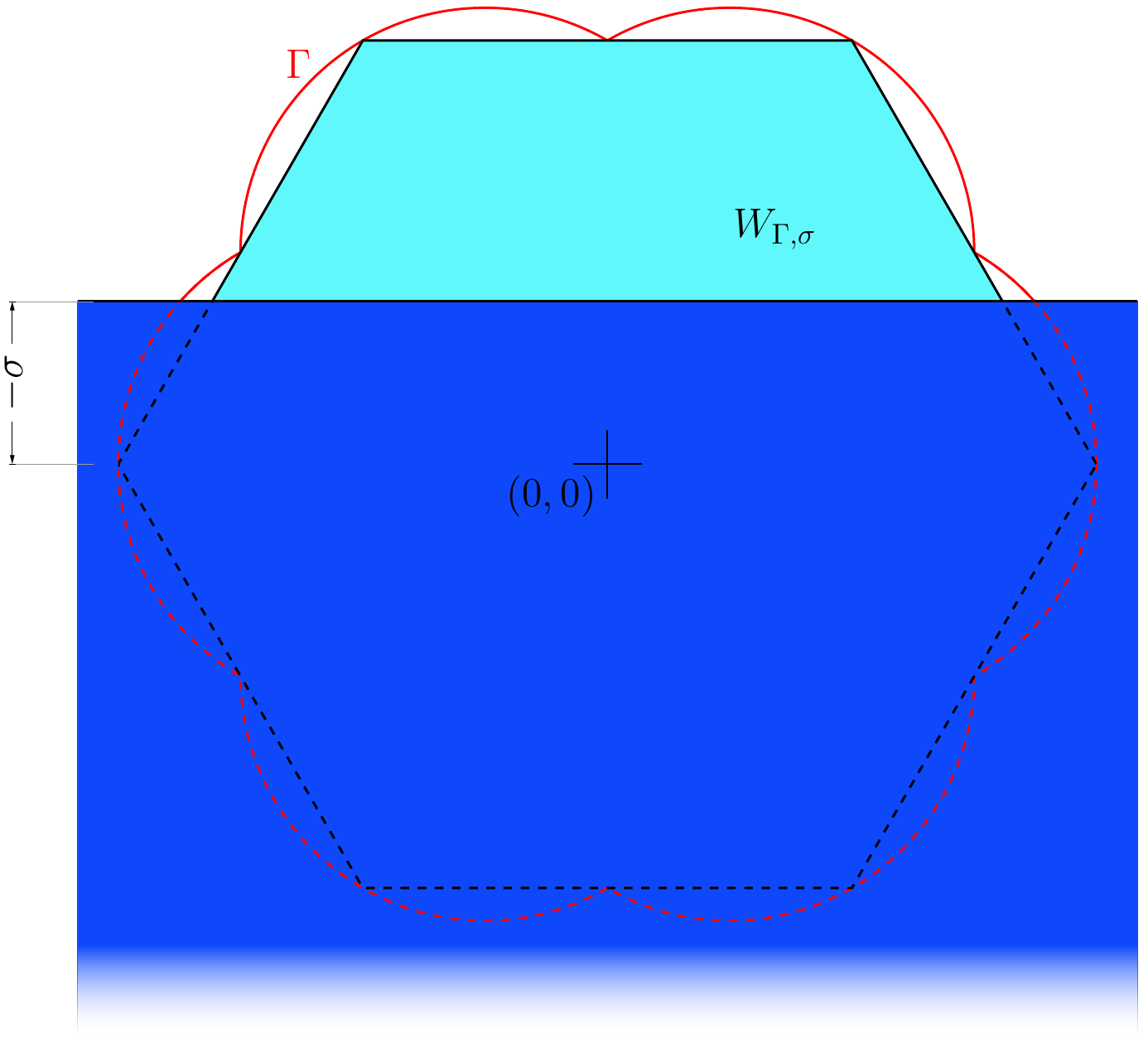}
\caption{\PPP Winterbottom construction for the minimizer of $ \mathcal{E}$, \PPP on the left for $\sigma>0$ and on the right for $\sigma<0$ \EEE (see  \cite{Winterbottom}). }
\label{fig: winterbottom}
\end{figure}

We intend here to generalize the analysis of  \cite{Yeung-et-al12} for $d=2$ to the situation of $S\neq\emptyset$ by taking into account at the discrete level also the atomic interactions of the particles of the crystalline drops with the particles of the substrate, which we allow to possibly belong to a different species of particles, and we suppose occupying all sites of a fixed reference lattice  $\mathcal{L}_S\subset S$.
 Film atoms are instead let free to move in a lattice $\mathcal{L}_F$ chosen to be triangular and contained in $\mathbb{R}^2\setminus \overline{S}$, so that admissible configurations of crystalline drops  with $n\in\Nz$ film atoms are  $D_n:=\{x_1,\ldots,x_n\}\subset\mathcal{L}_F$  (see Figure \ref{fig:lattices}). 
By adding the contribution $E_S:(\mathbb{R}^2\setminus \overline{S})^{n}\to\Rz\cup\{\infty\}$ to the energy of  \cite{Yeung-et-al12} to include atomic  interactions of film atoms with substrate atoms, the overall energy $V_n$ of an admissible configuration $D_n:=\{x_1,\ldots,x_n\}$ is given by 
$$
V_n(D_n)=V_n(x_1,\ldots,x_n):=E_F(D_n)+E_S(D_n),
$$
where $E_F:(\mathbb{R}^2\setminus \overline{S})^{n}\to\Rz\cup\{\infty\}$ represents the contribution of the atomic interactions among film atoms.  
More precisely, $E_F$ and $E_S$ are defined by
$$
E_F(D_n)=E_F(x_1,\ldots,x_n):=\UUU\sum_{i\neq j} v_{FF}(|x_i-x_j|)
$$
and 
$$
E_S(D_n)=E_S(x_1,\dots,x_n):=\sum_{i=1}^n \sum_{s\in\mathcal{L}_S} v_\textrm{FS}(|x_i-s|),
$$
respectively, where  $v_{F\alpha}$ for $\alpha=F,S$ are Heitmann-Radin sticky-disc two-body potentials attaining their minimum values $-c_\alpha$ at $e_\alpha>0$, where $e_\alpha$ is the distance between nearest neighbors in $\mathcal{L}_\alpha$  (see Figure \ref{fig:lattices}). 

\FFF We recall  that even with Heitmann-Radin potentials the crystallization of the minimizers of $V_n$ has been shown so far only in the case with $S=\emptyset$ in  \cite{Heitmann-Radin80} by showing that the minimizers of $E_F$ are subset of a triangular lattice.  
The  rigidity assumption of \EEE  prescribing reference lattices $\mathcal{L}_F$ and $\mathcal{L}_S$, \FFF besides imposing the \emph{non-interpenetration} for the film and substrate species of  atoms, which  remain separated by $\partial S$,   also entails that \EEE the elastic energy associated to the mismatch between the optimal crystalline lattices of the two materials  of the drop and the substrate at equilibrium 
\FFF is supposed to be \EEE all released by means of \PPP the \EEE \emph{periodic dislocations} \PPP of  the global reference lattice $\mathcal{L}:=\mathcal{L}_F\cup\mathcal{L}_S$  prescribed \EEE at the 
  film-substrate interface \FFF $\partial S$. \EEE 
 A study in which the complementary situation where elastic deformations of a homogeneous reference lattice $\mathcal{L}$ without dislocations between the film and the substrate are considered, is available in \cite{KrP}, where the linear-elastic models for epitaxially-strained thin films introduced in \cite{DP1,DP2,FFLM2, S2,spencer1997equilibrium} are derived from nonlinear elastic atomistic energies. 
 
  In our setting due to the periodic dislocations created at the interface between $\mathcal{L}_F$ and $\mathcal{L}_S$ the substrate interactions included in  $E_S$ \FFF are in general non-constant  (if not when $e_F$ is a multiple of $e_S$) and \EEE may result in \emph{periodic oscillations} between \FFF null and  negative  contributions  \EEE to the overall energy, referred to  in the following  as \emph{periodic adhesion deficit}. \FFF The \EEE presence of   
such oscillations \FFF substantiate \EEE the employment of \emph{homogenization techniques} for periodic structures \FFF (see \cite{AlbDes}  for the continuum setting), \EEE which represents one of the difference with the analysis carried out in \cite{Yeung-et-al12}. \FFF It then turns out though that \EEE the homogenized limit \FFF actually \EEE coincide  with the average \FFF in our setting  (and in \cite{CM} in the continuum setting). \EEE 

Moreover,  the \FFF periodic adhesion deficit  at the drop/substrate \EEE region \FFF induces  \EEE a \emph{lack of compactness} for (the properly scaled) energy-equi-bounded sequences  (even up to uniform translations), which is not treatable with only adopting \FFF local arguments \EEE at the substrate surface similar  to the one \FFF employed in \EEE \cite{Yeung-et-al12}.  In order to balance up the deficit we subdivide drop configurations  in    \emph{strips vertical to the substrate} so that enough boundary particles not adhering with the substrate (and so without deficit) are counted. Then, summing out all the strips allows to determine a \emph{global lower bound} to the overall surface contribution and to recover compactness in a proper subclass of admissible configurations, i.e., \emph{almost-connected configurations} (see Section \ref{sec:almost_connected}), that are 
 configurations union of connected components positioned at ``substrate-bond'' distance. Such limitation is then overcome by means of ensuring that mass does not escape on the infinite  substrate surface. 

 Another reason for the lack of  compactness \FFF with \EEE substrate interactions is the possibility for minimizing drop configurations to spread out on the infinite substrate surface forming an infinitesimal wetting layer, which for $E_S\not\equiv0$ can be actually favored. Therefore, a peculiar aspect of our analysis resides in distinguishing such \emph{wetting regime} from  the \emph{dewetting regime}. 
 More precisely, we characterize a \emph{dewetting threshold} \FFF in terms of \EEE the interatomic potentials $v_{FF}$ and $v_{FS}$, \FFF namely
 \begin{equation}\label{intro_dewetting_condition}
\begin{cases}
c_{S}< 4c_{F} & \text{if $e_F$ is a multiple of $e_S$,}\\
c_{S}< 6c_{F} & \text{otherwise,}
\end{cases}
\end{equation} 
under which the emergence of the minimizers of \eqref{winterbottomsurfaceenergy} with full $\Rz^2$-Lebesgue measure is shown. \EEE

\PPP The results of the paper are threefold (see Section \ref{sec: main_results}): The first result, Theorem \ref{wetting_theorem}, is a crystallization result for wetting configurations achieved by induction arguments in which the dewetting-threshold condition \FFF  \eqref{intro_dewetting_condition} \EEE
 is singled out by treating separately the situation of constant and non-constant substrate contributions\FFF. \EEE 
In this regard notice that the characterization of the dewetting regime coming from  continuum theories (see, e.g., \cite{BA}) does not represent in general a good prediction for the discrete setting due to \FFF the deficit averaging  effects taking \EEE place in the passage from discrete to continuum. More precisely, as described in \cite{BA} (with the extra presence of a gravity-term  perturbation of $\mathcal{E}$) the condition 
\begin{equation}\label{dewetting_continuum}
\sigma>-\Gamma(\nu_S) 
\end{equation}
is the natural requirement in the continuum ``ensuring that it is not energetically preferred for minimizers to spread out into an infinitesimally thin sheet''. However, condition \eqref{dewetting_continuum} coincides with the dewetting-threshold condition of the discrete setting only when $e_F$ is a multiple of $e_S$ (see \eqref{intro_sigma} and \eqref{intro_gamma} below), being otherwise the latter condition more restrictive. 

The second result, Theorem \ref{connectness},  provides a \emph{conservation of mass}  for the solutions of the discrete minimum problems
\begin{equation}\label{discrete_minimum}
\min_{D_n\subset\mathcal{L}_F}{V_n(D_n)}\
\end{equation}
as the number $n$ of atoms tends to infinity, which is crucial to overcome the lack of compactness outside the class of almost-connected sequences of energy-equibounded minimizers. 
In particular, it consists in proving that it is enough to select a connected component  among those with  largest cardinality for each solution of \eqref{discrete_minimum}. 
 This is achieved by  proving compactness   for  almost-connected energy minimizers  and then by defining a proper transformation $\mathcal{T}$ of configurations (based on iterated translations of connected components as detailed in Definition \ref{transformation}), which always allows to pass to an almost-connected sequence of minimizers.

The last result, Theorem \ref{thm:convergence_minimizers},  relates to the convergence  of the minimizers of \eqref{discrete_minimum}  as $n\to\infty$  to a minimizer of \eqref{winterbottomsurfaceenergy} in the family of crystalline-drop regions 
$$
\mathcal{D}_\rho:=\{D\subset\Rz^2\setminus S\,:\, \text{set   of finite  perimeter, bounded and such that  $|D|=1/\rho$}\}, 
$$
whose existence follows also from the proof, 
where $\rho$ is the atom density  in $\mathcal{L}_F$ per unit area.
Such convergence is obtained (up to extracting a subsequence and performing horizontal translations on the substrate $S$) as a direct consequence of the conservation of mass  provided by  Theorem \ref{connectness} and of a $\Gamma$-convergence result for properly defined versions of $V_n$ and $\mathcal{E}$ in the space $\mathcal{M}(\Rz^2)$ of Radon measures  on $\Rz^2$ 
with respect to the weak* convergence of measures  as the number $n$ of  film atoms tends to infinity. 

More precisely, we consider the one-to-one correspondence between drop configurations $D_n\subset\mathcal{L}_F$ and  their associated  empirical measures  $\mu_{D_n}\in\mathcal{M}(\Rz^2)$ (see definition at \eqref{empiricalmeasures}), introduce an energy $I_n$ defined in $\mathcal{M}(\Rz^2)$ such that
$$
I_n(\mu_{D_n})=V_n(D_n), 
$$
 and prove that the $\Gamma$-convergence of proper scalings  $E_n$ of  $I_n$, namely 
 $$
 E_n:=n^{-1/2}(I_n+6c_F n),
 $$
with respect to the weak* convergence of measures,  to a functional $I_\infty$ defined in such a way that 
 $$
 I_\infty\left(\rho\chi_D\right)=\mathcal{E}(D),
 $$
 for every sets of finite perimeter $D\subset\Rz^2\setminus S$ with $|D|=1/\rho$ and for specific \emph{effective expressions} of the surface tension $\Gamma$ and of the adhesivity $\sigma$ appearing in the  definition \eqref{winterbottomsurfaceenergy} of $\mathcal{E}$  in terms of the interatomic potentials $v_{FF}$ and $v_{FS}$. 
In particular, we obtain that
\begin{equation}\label{intro_sigma}
\sigma:=2c_F-\frac{c_S}{q}, 
\end{equation}
where $q$ relates to the proportion between $e_F$ and $e_S$ (see \eqref{eSratio}), and $\Gamma (\nu(\cdot))$ is found to be the $\pi/3$-periodic function  such that 
\begin{equation}\label{intro_gamma}
\Gamma (\nu(\varphi)):=2 c_F\left(\nu_2(\varphi)-\frac{\nu_1(\varphi)}{\sqrt{3}}\right)
\end{equation}
for every
$$
\nu(\varphi)=\left(\begin{array}{c} -\sin \varphi \\ \cos \varphi   \end{array} \right)
$$
with $\varphi \in [0, \pi/3]$. 

A crucial difference  with respect to  \cite{Yeung-et-al12} in the proof of the lower and upper bound of such $\Gamma$-convergence result  is that  the \FFF adhesion \EEE term in \eqref{winterbottomsurfaceenergy} can be negative and originates  in view of the \FFF averaging of the periodic adhesion deficit related to the dislocations at the film-substrate interface.  \EEE In particular, it is the limit of the adhesion portion of the boundary of auxiliary sets $H_n'$  associated to the configurations $D_n$ (see definition \eqref{H2} based on lattice \emph{Voronoi cells}) in the \emph{oscillatory sets} $O_n$ (see Figure \ref{fig:lattices}). 
We notice that for such averaging  arguments extra care is needed, as the results available from the continuum theories cannot directly be applied to the auxiliary sets $H_n'$ when $e_F$ is not a multiple of $e_S$, e.g., with respect to \cite{BA} (see also \cite{CM}) because \FFF of the non-constant  deficit, \EEE and  with respect to \cite{AlbDes} when $4c_F\leq c_S< 6c_F$ because of the discrepancy between the continuum and the discrete dewetting conditions.  

\FFF Finally, we observe that as a consequence of the $\Gamma$-convergence result contained in Theorem \ref{thm:convergence_minimizers}, not only the convergence of global minimizers of \eqref{discrete_minimum}  to a minimizer of \eqref{winterbottomsurfaceenergy}, but also the convergence of \emph{isolated local} (with respect to a proper topology) minimizers of $E_n$ to an \emph{isolated local} minimizer of \eqref{winterbottomsurfaceenergy} as $n\to\infty$ is entailed (see, e.g., \cite{kohn1}).   In this regard, we refer to  \cite{srolovitz1,srolovitz2} for the importance  of detecting the \emph{local equilibrium shapes} related to energies of type \eqref{winterbottomsurfaceenergy},  especially  in relation to the various kinetic phenomena affecting the  dewetting dynamics, such as Rayleigh-like instabilities, corner-induced instabilities, and periodic mass-shedding. \EEE

\subsection{\FFF Paper organization} \EEE In Section \ref{setting} we introduce the mathematical setting with  the discrete models (expressed  \FFF  both \EEE with respect to \FFF lattice configurations and to \EEE Radon measures) and the  continuum model, and the three main theorems of the paper. In Section \ref{sec:wetting} we treat the wetting regime and prove  Theorem \ref{wetting_theorem}. In Section \ref{sec:compactness} we establish  the compactness result for energy-$E_n$-equibounded almost-connected sequences. In Section \ref{sec:lower_bound} we prove the lower bound \FFF of the \EEE $\Gamma$-convergence result. In Section \ref{sec:upper_bound} we prove the upper bound \FFF of \EEE the $\Gamma$-convergence result. In Section \ref{sec:main_results} we study the convergence of almost-connected transformations  of minimizers  and present the proofs of both Theorems \ref{connectness} and \ref{thm:convergence_minimizers}. 
Finally, an Appendix with specific auxiliary results particularly important in the various  proofs is added for the Reader's convenience.

\EEE

\section{Mathematical setting \PPP and main results} \label{setting}
\PPP
In this section we rigorously introduce the discrete and continuous models, the notation and definitions  used throughout the paper, and the main results. 

 


\subsection{\PPP Setting with lattice configurations}\label{sec:lattice_configurations}   We begin by introducing a reference lattice  $\mathcal{L}\subset\Rz^2$ for the atoms of the substrate and of the film, which we assume to remain separate. 
We define  $\mathcal{L}:=\mathcal{L}_S\cup\mathcal{L}_F$, where $\mathcal{L}_S\subset\overline{S}$ denotes the reference lattice for the substrate atoms, $S:=\Rz\times\{r\in\Rz\,:\, r<0\}$ is referred to as the  \emph{substrate region},  and  $\mathcal{L}_F\subset\Rz^2\setminus\overline{S}$ is the reference lattice for the film atoms. 

More precisely, we consider the substrate lattice as a fixed lattice, i.e., every lattice site in $\mathcal{L}_S$ is occupied by a substrate atom,  such that 
$$\partial\mathcal{L}_S:=\mathcal{L}_S\cap\{(r,0):\,r\in\Rz\,\}=\{s_k:=(k e_S,0) \,:\,k\in\Zz\}$$
for a positive lattice constant $e_S$, and we refer to $\partial\mathcal{L}_S$ as to the  \emph{substrate surface} (or \emph{wall}). 
For the film lattice   $\mathcal{L}_F$ we choose a triangular lattice with parameter $e_F$ normalized to 1, namely 
$$\mathcal{L}_F:=\{x_F+k_1 \tone+k_2 \ttwo\,:\, \UUU \textrm{$k_1\in\Zz$ and $k_2\in\Nz\cup\{0\}$} \BBB\}
$$
where  $x_F:=(0,e_S)$, \EEE
$$\tone:={1 \choose 0},\quad\textrm{and}\quad\ttwo:= \UUU  \frac{1}{2}{1 \choose \sqrt{3}}. 
$$
We  denote by $\partial \mathcal{L}_F$ the \emph{lower boundary} of the film lattice, i.e.,
$$
\partial\mathcal{L}_F:=\{x_F+k_1 \tone:\, \UUU \textrm{$k_1\in\Zz$} \BBB\}
$$
\PPP and by  $\partial \mathcal{L}_{FS}$ the collection of sites  in the lower boundary of the film lattice at a distance of $e_S$ from an atom in $\partial\mathcal{L}_S$, i.e., $$\partial \mathcal{L}_{FS}:= \partial \mathcal{L}_F \cap \left( \partial \mathcal{L}_S +e_S\tthree \right)$$ 
where 
$$
\tthree:={0 \choose 1}
$$
(see Figure \ref{fig:lattices}).
\begin{figure}
\includegraphics[width=0.75\textwidth]{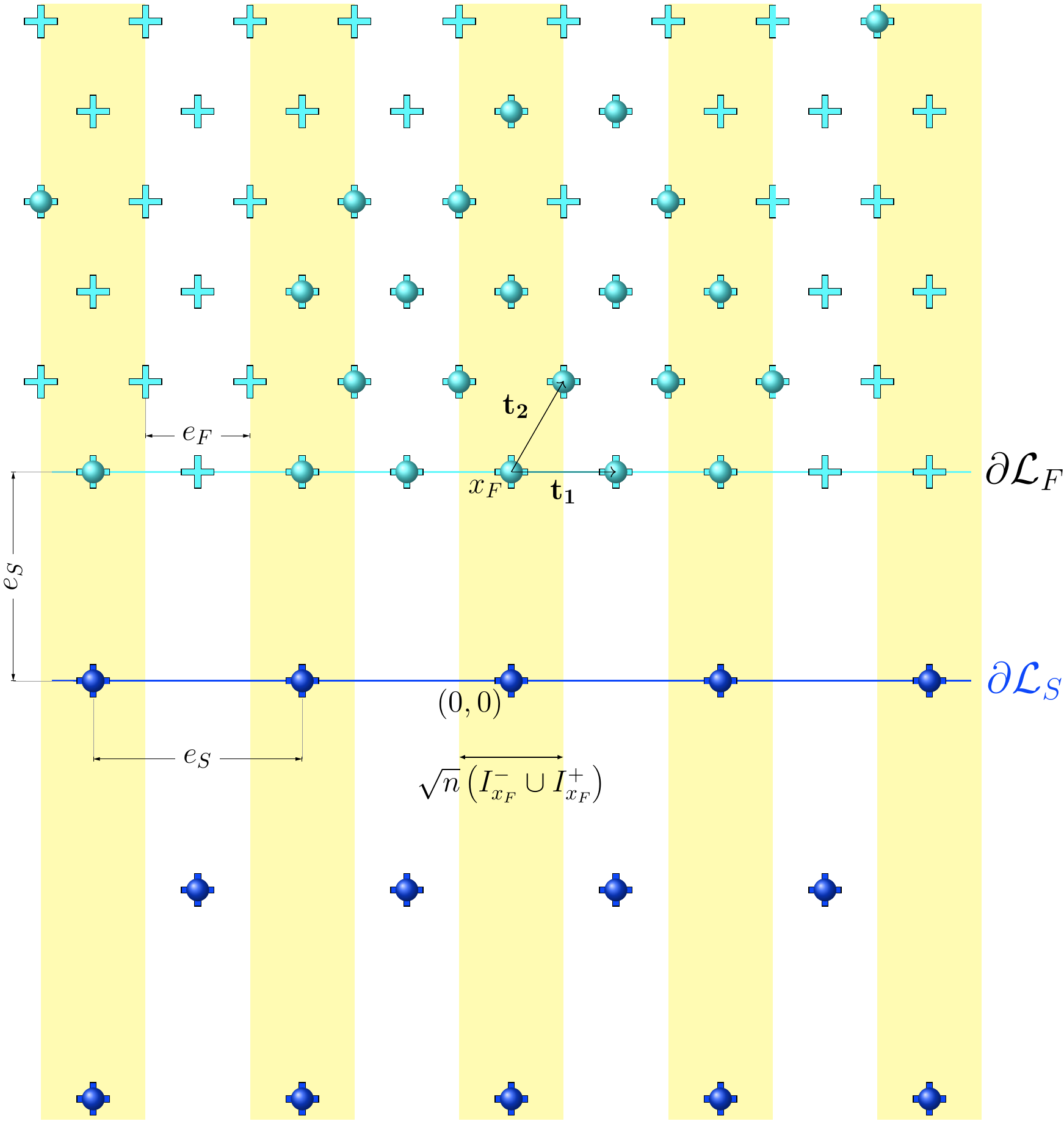}
\caption{\PPP A portion of the lattices $\mathcal{L}_F$  and $\mathcal{L}_S$ is depicted with the respective lattice sites in light and dark blue crosses, respectively. The lattice $\mathcal{L}_S$ is fully occupied by substrate atoms represented by dark blue balls, while only some sites of $\mathcal{L}_F$ are occupied by  film atoms represented by light blue balls.  The ``interface'' $\partial \mathcal{L}_{F}$ consists of all the lattice sites on the light-blue line, while the ``interface'' $\partial \mathcal{L}_S$ consists of all the lattice sites on the dark-blue line. In yellow we can see the \emph{oscillatory set}  related to the lattice sites in $\partial \mathcal{L}_{FS}$, which is  introduced in Section \ref{sec:lower_bound}. 
}
\label{fig:lattices}
\end{figure}
 The sites of the film lattice are not assumed to be completely filled and we refer to a set of $n\in\Nz$ sites $x_1,\dots,x_n\in \mathcal{L}_F$ occupied by film atoms as a \emph{crystalline configuration} denoted by $D_n:=\{x_1,\dots,x_n\}\subset \mathcal{L}_F$. Notice that the labels for  the elements of a configuration $D_n$ are uniquely determined by increasingly assigning them with respect to a chosen fixed order \FFF on \EEE the lattice sites \FFF of \EEE $\mathcal{L}_F$.  With a slight abuse of notation we refer to $x\in D_n$ as an atom in $D_n$ (or in $\mathcal{L}_F$). We denote the family of crystalline configurations with $n$ atoms by $\mathcal{C}_n$. \FFF Furthermore, given a set $A\subset \mathbb{R}^2$, its cardinality is indicated by $\#A$, so that $$\mathcal{C}_n:=\{A\subset\mathcal{L}_F\,:\, \#A=n\}.$$ \EEE

For every atom $x\in\mathcal{L}_F$ we take into account both its atomistic interactions with other film atoms and with the substrate atoms, by considering the two-body atomistic potentials $v_{FF}$ and $v_{FS}$, respectively. We restrict to first-neighbor interactions and we define $v_{F\alpha}$ for $\alpha:=F,S$ as 
$$
v_{F\alpha}(r):=v_{\alpha}(r)
$$
where $v_{\alpha}$ are the Heitmann-Radin potentials defined for $\alpha:=F,S$ by 
\begin{equation}\label{HR}
v_{\alpha}(r):=\begin{cases} +\infty &\mbox{if }  r<e_{\alpha},\\ 
-c_{\alpha} &\mbox{if }  r=e_{\alpha},\\ 
0 & \mbox{if }  r>e_{\alpha}. 
\end{cases}
\end{equation}
with  $c_{\alpha}$ and $e_{\alpha}$ positive constants. 

 In the following, we refer to  \emph{film} and \emph{substrate neighbors}   of an atom $x$ in a configuration $D_n$ as to those atoms   in  $D_n$ at distance 1 from $x$, and to those atoms in  $\mathcal{L}_S$ at distance $e_S$  from $x$, respectively. Analogously, we refer to  \emph{film} and \emph{substrate bonds} of an atom $x$ in a configuration $D_n$ as to those segments connecting $x$ to its film and substrate neighbors, respectively. 
We also refer to the union of the closures of all film bonds of atoms in a configuration $D_n$ as the \emph{bonding graph} of  $D_n$, and we say that a crystalline configuration $D_n$ is \emph{connected} if every $x$ and $y$ in $D_n$ are connected through a path in the bonding graph of $D_n$, i.e., there exist $\ell\leq n$ and $x_k\in D_n$ for $k:=1,\dots,\ell$ such that $|x_k-x_{k-1}|=1$, $x_1=x$, and $x_\ell=y$. \FFF Moreover,  we define  the \emph{boundary of a configuration} $D_n\in\mathcal{C}_n$ as the set  $\partial D_n$ of atoms of $D_n$ with less than 6 film neighbors. \FFF We notice here that with a slight abuse of notation, given a  set $A\subset\mathbb{R}^2$  the notation $\partial A$ will also denote the topological boundary of a set $A\subset\mathbb{R}^2$  (which we intend to be always the way to interpret the notation when applied not to configurations in $\mathcal{C}_n$, or to lattices, such as for $\partial \mathcal{L}_{S}$, $\partial \mathcal{L}_{F}$, and $\partial \mathcal{L}_{FS}$). \EEE 

\EEE



The energy $V_n$ of a configuration $D_n:=\{x_1,\ldots,x_n\}\subset \mathcal{L}_F$ of $n$ particles is defined by 
\begin{equation}\label{V}
\UUU V_n(D_n)=\EEE V_n(x_1,\ldots,x_n)\PPP:=\UUU\sum_{i\neq j} v_{FF}(|x_i-x_j|)\,+\,  E_S(x_1,\ldots,x_n)\EEE
\end{equation}
where  $E_S:(\Rz^2\setminus\overline{S})^{n}\EEE\to\Rz\cup\{\infty\}$ represents the overall contribution of the substrate interactions defined  as  
\begin{equation}\label{substrateenergy}
E_S(D_n)=E_S(x_1,\dots,x_n):=\sum_{i=1}^n v^1(x_i),
\end{equation}
where the one-body potential $v^1$ is defined by
\begin{equation}\label{substrate2}
v^1(x):=\sum_{s\in\mathcal{L}_S} v_\textrm{FS}(|x-s|)
\end{equation}
for any $x\in\Rz\times\{r\in\Rz\,:\, r>0\}$.  Notice that from the definition of $v_{FS}$ and $x_F$ for any $x\in\mathcal{L}_F$ the sum in \eqref{substrate2} is finite and  
$$v^1(x)\in\{0,-c_S\}.$$

\PPP In the following we will always focus on the  case
\begin{equation}\label{eSratio}
e_S:=\frac{q}{p}
\end{equation}
for some $p,q\in\Nz$ without common factors, since the case of $e_S=re_F$ for some $r\in\Rz\setminus\Qz$ is simpler, as the contribution of $E_S$ is negligible.  More precisely, for $e_S=re_F$ with $r\in\Rz\setminus\Qz$ the same analysis (or the one in \cite{Yeung-et-al12})   applies, and, up to rigid transformations, minimizers converge to a Wulff shape in $\Rz^2\setminus S$ with the Wulff-shape boundary intersecting $\partial S$ at least in a point.

We also notice that the setting in which $\mathcal{L}_F$ is replaced by 
$$\mathcal{L}'_F:=\mathcal{L}_F+\left(\frac{e_S}{2},\frac{\sqrt{3}e_S}{2}-e_S\right)$$
where each atom in $\partial\mathcal{L}_{FS}$  may present up to two   (instead of one) bonds with $\mathcal{L}_S$ (which are obliques instead of in $\tthree$ direction)  is analogous and the same arguments of the paper lead to the corresponding  $\Gamma$-convergence result. 

\PPP 
\subsection{Setting with Radon measures} \label{radon_model} 
The $\Gamma$-convergence result is established  for a version of the previously described discrete model expressed in terms of \emph{empirical measures} since it is obtained with respect to the weak* topology of Radon measures \cite{AFP}. We denote the space of Radon measures on $\Rz^2$ by $\mathcal{M}(\Rz^2)$ and we write $\mu_{n}\stackrel{*}{\rightharpoonup} \mu$ to denote
the  convergence of  a sequence  $\{\mu_n\}\subset\mathcal{M}(\Rz^2)$ to a measure $\mu\in\mathcal{M}(\Rz^2)$ with respect to the weak* convergence of measures. The empirical measure $\mu_{D_n}$ \emph{associated to a configuration} $D_n:=\{x_1,\dots,x_n\}\in\mathcal{C}_n$ is defined by 
 \begin{equation} \label{empiricalmeasures}
 \mu_{D_n}:=\frac{1}{n} \sum_{i=1}^n \delta_{\frac{x_i}{\sqrt n}}, 
\end{equation}
\FFF where $\delta_z$ represents the Dirac measure concentrated at a point $z\in \Rz^2$, \EEE and the family of empirical measures related to configurations in $\mathcal{C}_n$ is denoted by $\mathcal{M}_n$, i.e.,
\begin{equation}\label{radon_space}
\mathcal{M}_n:=\{\mu\in\mathcal{M}(\Rz^2)\ :\ \text{there exists $D_n\in\mathcal{C}_n$ such that $\mu=\mu_{D_n}$}\}.
\end{equation}

The  functional $I_n$  associated to the configurational energy $V_n$ and expressed in terms of Radon measures is given by 
\begin{equation}
 I_n(\UUU\mu\EEE)\PPP:=
 \EEE\left\{\begin{array}{lll} 
  \int_{(\PPP\mathbf{R}^2\setminus\overline{S}\EEE)^2 \backslash \textrm{diag}} n^2 v_{FF}(n^{1/2}|x-y|) d \mu\UUU(x)\EEE 
  \otimes d \mu\UUU(y) &
  &\quad\textrm{\PPP if $\mu\in\mathcal{M}_n$, } 
  \\ 
&\hspace{-15ex} \RRR \EEE + \int_{\PPP\mathbf{R}^2\setminus\overline{S}\EEE} n \UUU v^1\EEE (n^{1/2}x) d \UUU \mu(x) \EEE &  \label{radon_functional}\\
& \\  +\infty &&\quad\textrm{otherwise,} 
\end{array}\right.
\end{equation}
\FFF where \EEE
$$ 
\textrm{diag}:=\{(y_1,y_2) \in \Rz^2: y_1=y_2\}. 
  $$	

\PPP We notice that  the two versions of the discrete model are equivalent, since 
\begin{equation}\label{energy_equivalence}
V_n(D_n)=I_n(\mu_{D_n})
\end{equation}
for every  configuration $D_n\in\mathcal{C}_n$, where $\mu_{D_n}\in\mathcal{M}_n$ is defined by \eqref{empiricalmeasures}, and that $D_n$ minimizes $V_n$ among crystalline configurations in $\mathcal{C}_n$ if and only if $\mu_{D_n}$ minimizes $I_n$  among Radon measures of $\mathcal{M}(\Rz^2)$.

\PPP \subsection{\PPP Local and strip energies} \label{local_energy} \EEE We  define a \emph{local energy} \PPP $E_{\rm loc}$  per site $x\in\mathcal{L}_F$ with respect to a configuration $D_n$, \EEE by
 \begin{equation}\label{Eloc}
 E_{\rm loc}(x):=\begin{cases}
\sum_{y\in D_n\setminus\{x\}} v_{FF}(|x-y|) \,+\,6c_F 
 &\text{if $x\in D_n$,}\\
 0 &\text{if $x\notin D_n$},
 \end{cases}
 \end{equation}
\PPP which corresponds in the case of an atom $x\in D_n$ to the number of missing film bonds of $x$.
 We also refer to deficiency $ E_{\rm def}(x)$ of a site $x\in\mathcal{L}_F$ with respect to a configuration $D_n$ as to the quantity
  \begin{equation}\label{deficiency}
 E_{\rm def}(x):=  
 \begin{cases}
 E_{\rm loc}(x)\,+\,v^1(x)
 &\text{if $x\in D_n$,}\\
 0 &\text{if $x \notin  D_n$}.
 \end{cases}
  \end{equation}
Furthermore, we define the \emph{strip} $\mathcal{S}(x)$ associated to any lattice site \PPP $x:=(x^1,e_S)\in D_n\cap\partial\mathcal{L}_{FS}$ with $x_1\in\Rz$ \EEE as the collection of atoms 
\begin{equation}\label{strip}
\mathcal{S}(x)=\mathcal{S}_{D_n}(x):=\{x,x_{\pm},\tilde{x},\tilde{x}_{\pm}\}\cap D_n
\end{equation}
where \PPP $x_{\pm}$, \EEE $\tilde{x}$,  and $\tilde{x}_{\pm}$ are defined by
\begin{align*}
&x_{\pm}:=x\pm\tone,\\
&\tilde{x}:=(x^1,y_M)\quad \text{where}\quad y_M:=\max\{y\geq0\,:\, (x^1,y)\in D_n\},\\
&\tilde{x}_{+}:=\tilde{x}+\ttwo,\\
&\tilde{x}_{-}:=\tilde{x}+\ttwo-\tone
\end{align*}
\PPP (see Figure \ref{fig:strip}). \EEE
\begin{figure}
\includegraphics[width=0.4\textwidth]{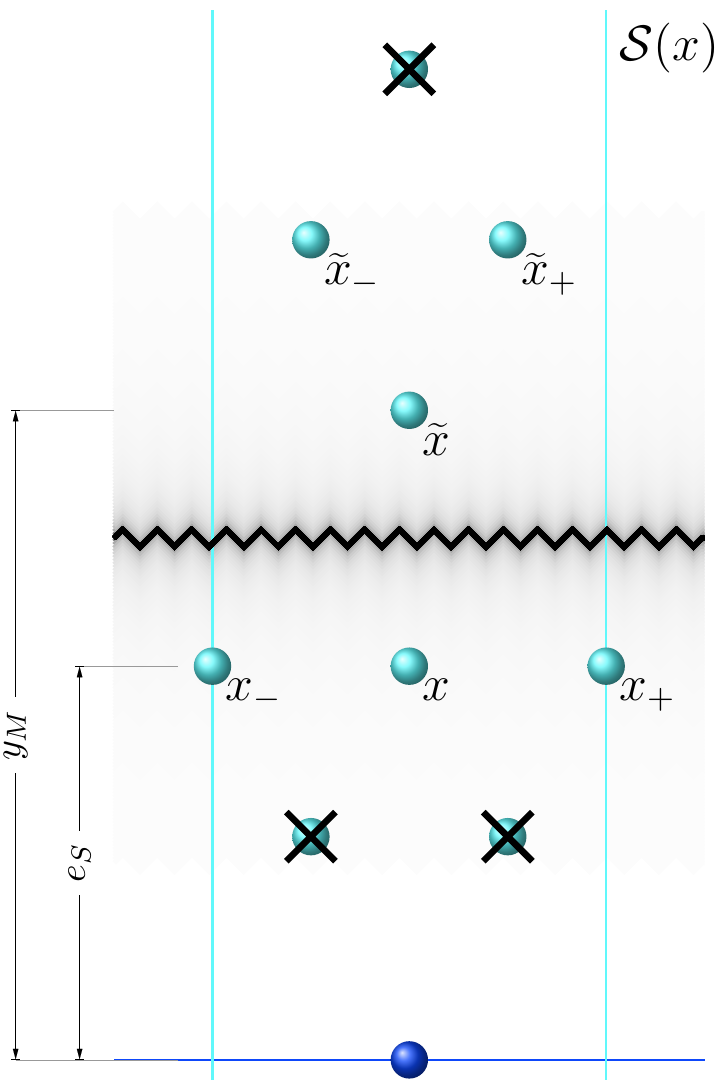}
\caption{\FFF The \EEE strip $\mathcal{S}(x)$ centered at an atom $x\in\partial\mathcal{L}_{FS}$ of a crystalline configuration $D_n$ is depicted as an example of a strip containing all the elements $x,x_{\pm},\tilde{x},\tilde{x}_{\pm}$ \PPP  with the possibility of the \emph{strip center} $x$ and the \emph{strip top} $\tilde{x}$ to coincide if $y_M=e_S$. The sites  indicated by crossed atoms are sites of the planar lattice $\{x_F+k_1 \tone+k_2 \ttwo\,:\, \textrm{$k_1, k_2\in\Zz$}\}$ that surely are not in $D_n$ by definition of $\mathcal{L}_F$ and  $\mathcal{S}(x)$.}
\label{fig:strip} 
\end{figure}
In the following we refer to $x$ as the \emph{strip center} of $\mathcal{S}(x)$, to $x_{\pm}$ as \PPP the \EEE \emph{strip lower (right and left) sides}, to $\tilde{x}$ as  \PPP the \EEE  \emph{strip top}, and to $\tilde{x}_{\pm}$ as  \PPP the \EEE \emph{strip above (right and left) sides}.  Note \PPP that $x$ and $\tilde{x}$ coincide if $y_M=e_S$.  

We define the \emph{strip energy} associated to a strip $\mathcal{S}(x)$ by 
\begin{equation}\label{Estrip}
  E_{\rm strip}(x):= E_{\rm strip, below}(x)\,+\,E_{\rm strip, above}({\color{black} x}),
\end{equation}
where
 \begin{align} \label{Estrip_below1}
 E_{\rm strip, below}(x):=& E_{\rm loc}(x)\,+\, \frac12 E_{\rm loc}(x_{+})\,+\, \frac12 E_{\rm loc}(x_{-})-c_S
   \end{align}
  in the case $\PPP q\neq 1 \EEE$, while
  \begin{align} \label{Estrip_below2}
 E_{\rm strip, below}(x):=&\frac12 E_{\rm loc}(x)\,+\, \frac14 E_{\rm loc}(x_{+})\,+\, \frac14 E_{\rm loc}(x_{-})-c_S
 \end{align}  
 in the case $\PPP q= 1 \EEE$, and    
\begin{equation}  \label{Estrip_above}
 E_{\rm strip, above}(x):=\begin{cases}
E_{\rm loc}(\tilde{x})\,+\, w_-(\tilde{x})( E_{\rm loc}(\tilde{x}_{+})\,+\, w_+(\tilde{x})  E_{\rm loc}(\tilde{x}_{-}) &\textrm{if $\tilde{x}\neq x$,}\\
w_-(\tilde{x}) E_{\rm loc}(\tilde{x}_{+})\,+\, w_+(\tilde{x}) E_{\rm loc}(\tilde{x}_{-}) &\textrm{if $\tilde{x}=x$}
  \end{cases}
\end{equation}
with weights $w_{\pm}(\tilde{x})\in\FFF \{1/2,1\}$ \EEE
given by
\begin{equation}  \label{weigths_above}
w_{\pm}(\tilde{x}):=\begin{cases}
1&\textrm{if  $x_{\pm}\not\in\PPP D_n\cap\partial\mathcal{L}_{FS}$ or \EEE $\tilde{x}_{\pm}\neq \widetilde{(x_{\pm})}_{\mp}$,}\\
\frac12 &\textrm{if  $x_{\pm}\in\PPP D_n\cap\partial\mathcal{L}_{FS}$ and \EEE$\tilde{x}_{\pm}= \widetilde{(x_{\pm})}_{\mp}$}.
  \end{cases}
\end{equation}

  \subsection{\PPP Almost-connected configurations} \label{sec:almost_connected} \PPP 
   
   We recall from Section  \ref{sec:lattice_configurations} that  a configuration $D_n$ is said to be connected  if every $x$ and $y$ in $D_n$ are connected through a path in the bonding graph of $D_n$, i.e., there exist $\ell\leq n$ and $x_k\in D_n$ for $k:=1,\dots,\ell$ such that $|x_k-x_{k-1}|=1$, $x_1=x$, and $x_\ell=y$, and we refer to maximal bonding subgraphs of $D_n$ connected through a path as \emph{connected components} of $D_n$. 

  In order to treat the situation when $q\neq1$ we need to introduce also a weaker notion of  connectedness of configurations, which depends on $e_S$: We say that a configuration $D_n$ is \emph{almost connected} if it is connected when $q=1$, and, if \FFF there exists an  enumeration of its $k:=k_{D_n}$ connected components, say   $D_n^i$, $i=1,\dots,k$, such that \FFF each \EEE $D_n^{i}$ is separated by at most $q$  from $\cup_{l=1}^{i-1}D_n^l$  for every $i=2,\dots,n$,  when $q\neq1$. 
We say that a family of connected components of $D_n$ form an \emph{almost-connected component} of $D_n$ if their union is almost connected and, if $q\neq1$, it is distant from all other components of $D_n$ by more than $q$. 

\PPP
\begin{definition}\label{transformation} 
Given a configuration $D_n\in\mathcal{C}_n$, 
we define the \emph{transformed configuration} $\mathcal{T}(D_n)\in\mathcal{C}_n$ of $D_n$ as 
$$
\mathcal{T}(D_n):=\mathcal{T}_2(\mathcal{T}_1(D_n)),
$$ 
where $\mathcal{T}_1(D_n)$ is the configuration resulting by iterating the following procedure, starting from $D_n$:
 \begin{itemize}
 \item[-] If there are connected components without any activated bond with an atom of $\partial\mathcal{L}_S$, then select one of those components with lowest distance from $\partial\mathcal{L}_S$; 
 \item[-] Translate the component selected at the previous step of a vector in direction $-\ttwo$ till either a bond with another connected component or with the substrate is activated. 
 \end{itemize}
 (notice that the procedure ends when all connected components of $\mathcal{T}_1(D_n)$ have at least a bond with  $\partial\mathcal{L}_S$), and  $\mathcal{T}_2(\mathcal{T}_1(D_n))$ is the configuration resulting by iterating the following procedure, starting from  $\mathcal{T}_1(D_n)$: 
  \begin{itemize}
   \item[-] If there are more than one almost-connected component, then select the almost-connected component whose  leftmost bond with  $\partial\mathcal{L}_S$ is the second (when compared with \FFF the \EEE other almost-connected components) starting from the left;
       \item[-] Translate the  almost-connected component selected at the previous step of a vector  $-kq\tone$ for some $k \in \mathbf{N}$ till, if $q=1$,  a bond with another connected component is activated, or, if $q\neq1$, the distance with another almost-connected component is  less or equal to $q$; 

  \end{itemize} 
 (notice that the procedure ends when $\mathcal{T}_2(D_n)$ is almost connected). 
\end{definition}

\noindent We notice that the transformed configuration $\mathcal{T}(D_n)$ of a configuration $D_n\in\mathcal{C}_n$ satisfies the following properties: 
 \begin{itemize}
 \item[(i)] $\mathcal{T}(D_n)$ is almost connected;
  \item[(ii)]  Each  connected component of $\mathcal{T}(D_n)$ includes at least an atom bonded to $\partial\mathcal{L}_S$;
  \item[(iii)] $V_n(\mathcal{T}(D_n))\leq V_n(D_n)$ (as no active bond of $D_n$ is deactivated by performing the transformations $\mathcal{T}_1$ and $\mathcal{T}_2$);
  \end{itemize}
  and,  if $D_n$ is a minimizer of $V_n$ in $\mathcal{C}_n$, then 
   \begin{itemize}
 \item[(iv)] $\mathcal{T}_1(D_n):=D_n$;
  \item[(v)] $\mathcal{T}$ consists of translations of the almost-connected components of $D_n$ with respect to a vector (depending on the component)  in the direction $-\tone$ with norm  in $\Nz\cup\{0\}$. 
  \end{itemize}
  
  \FFF Finally we also observe that the definitions of $\mathcal{T}_1$, $\mathcal{T}_2$, and $\mathcal{T}$ are independent from $n$. 

\EEE

  \subsection{\PPP Continuum setting} \label{sec:continuum_model} \EEE

\PPP For every set of finite perimeter $D\subset\Rz^2\setminus S$ we define its anisotropic surface energy $\mathcal{E}$ by \EEE
\begin{equation} \label{continuum}
\mathcal{E} (D):= \int_{\partial^* D \cap \partial S} \Gamma (\nu_D) d \mathcal{H}^1+\left( 2c_F-\frac{c_S}{q}\right) \mathcal{H}^1(\partial^*D \cap \partial S)
\end{equation}
\PPP where $\partial^* D$ denotes the reduced boundary of $D$ and the \emph{anisotropic surface tension} $\Gamma:\mathbb{S}^{1} \to \mathbb{R}$ is the function such that it holds 
\begin{equation} \label{gamma}
\Gamma (\nu(\varphi))=2 c_F\left(\nu_2(\varphi)-\frac{\nu_1(\varphi)}{\sqrt{3}}\right)
  \end{equation} 
 for every
$$
\nu(\varphi)=\left(\begin{array}{c} -\sin \varphi \\ \cos \varphi   \end{array} \right)\in\mathbb{S}^{1}\quad\text{with}\quad  \varphi \in \left[0, \frac{\pi}{3}\right],
$$
\FFF also when $\Gamma\circ\nu$ is extended \EEE periodically on $\Rz$ as a  $\pi/3$-periodic function. Notice  that  $\Gamma (\tthree)=\Gamma (\ttwo)=2c_F$. \FFF  By extending $\Gamma$ by homogeneity we obtain a convex function, and in particular  a  Finsler norm on $\Rz^2$. \EEE

\PPP We also use the following \emph{auxiliary surface energy} \FFF depending on $n$ \EEE in the proofs 
\begin{equation} \label{continuum_n}
\mathcal{E}_n (D):= \int_{\partial^* D \cap (\Rz^2\setminus \overline{S_n})} \Gamma (\nu_D) d \mathcal{H}^1+\left( 2c_F-\frac{c_S}{q}\right) \mathcal{H}^1(\partial^*D \cap \partial S_n) \end{equation}
where  
\begin{equation}\label{substrate_n}
S_n:=S+\frac{e_S}{\sqrt n}\tthree.
 \end{equation}

  \subsection{\PPP Main results} \label{sec: main_results} \EEE

\PPP In this section the rigorous statements of the main theorems of the paper are presented. We begin with following result that characterizes  the wetting regime in terms of a condition only depending on $v_{FF}$ and $v_{FS}$, and  the  minimizers in such regime. \EEE
 \PPP
 \begin{theorem}[Wetting regime]\label{wetting_theorem}
Let   $D^{\rm w}_n:=\{w_1,\dots,w_n\}\subset\partial\mathcal{L}_{FS}$ be  \FFF any configuration such \EEE that, if  $q= 1$, 
\begin{equation}\label{connected_wetting00}
w_{i+1}:=w_i+\tone
\end{equation}
for every $i=1,\dots,{n}$ and every $n\in\Nz$. \FFF It holds that \EEE  $D^{\rm w}_n$ \FFF satisfies \EEE the following two assertions for every $n\in\Nz$: 
\begin{itemize} 
	\item[(i)] $V_n(D^{\rm w}_n)=\min{V_n(D_n)}$,
	\item[(ii)] \PPP $V_n(D^{\rm w}_n)<V_n(D_n)$ \EEE for  \PPP every crystalline configuration $D_n$ with $D_n\setminus\partial\mathcal{L}_{FS}\neq\emptyset$ \emph{(}and, for the case  $q= 1$, also for every configuration $D_n$ with $D_n\setminus\partial\mathcal{L}_{FS}=\emptyset$ and for which \eqref{connected_wetting00}  does not hold\emph{)}\FFF,\EEE
\end{itemize}
if and only if  
\begin{equation}\label{wetting_condition_total}
\begin{cases}
c_{S}\geq 6c_{F} & \text{if $q\neq 1$},\\
c_{S}\geq 4c_{F} & \text{if $q=1$}. 
\end{cases}
\end{equation}
\FFF In particular, for the necessity of \eqref{wetting_condition_total} it is enough assertion (i), and more specifically that 
there exists an increasing subsequence $(n_k)_{k \in \mathbf{N}}$ such that $V_{n_k}(D^{\rm w}_{n_k})=\min{V_{n_k}(D_{n_k})}$ holds for every $n_k$. \EEE	
\end{theorem} 

We refer to  \eqref{wetting_condition_total} as a \emph{wetting condition} or as the \emph{wetting regime}, and to the opposite condition, namely 
\begin{equation}\label{dewetting_condition}
\begin{cases}
c_{S}<6 c_{F} & \textrm{if {\color{black} $\PPP q\neq 1 \EEE$}},\\
c_{S}< 4 c_{F} & \textrm{if {\color{black} $\PPP q= 1$}}, \EEE
\end{cases}
\end{equation}
as the \emph{dewetting condition} or the \emph{dewetting regime}. \EEE 
\PPP The following result shows that connected components with the largest cardinality of minimizers incorporate the whole mass in the limit. \EEE

\begin{theorem}[Mass conservation] \label{connectness} 
\PPP Assume \eqref{dewetting_condition}. If $ \widehat{D}_n$ are  minimizers of $V_n$ among all crystalline configurations in $\mathcal{C}_n$, i.e.,  
$$V_n(\widehat{D}_n)=\min_{D_n\in\mathcal{C}_n}{V_n(D_n)},$$
and we select for every $\widehat{D}_n$ a connected component $\widehat{D}_{n,1}\subset \widehat{D}_n$ with  largest cardinality, then
$$
\lim_{n \to \infty}{\mu_{\widehat{D}_n} (\widehat{D}_{n,1})}=1, 
$$
where $\mu_{\widehat{D}_n}$ are the empirical measure associated to $\widehat{D}_n$ defined by \eqref{empiricalmeasures}. \EEE
\end{theorem}	

\PPP We rigorously prove by $\Gamma$-convergence that the discrete models converge to the continuum model, and in view of the previous result (even in the lack of a direct compactness result for general sequences of minimizers, possibly not almost connected), we prove convergence (up to passing to a subsequence and up to translations) of the minimizers of the discrete models to a bounded minimizer of the continuum model, which in turns it is also proven to exist. \EEE

\begin{theorem}[Convergence of Minimizers]\label{thm:convergence_minimizers}
	\PPP Assume \eqref{dewetting_condition}. The following statements hold:
	\begin{itemize}
	\item[1.] The functional 
	\begin{equation}\label{converging_energies00}
	E_n:= n^{-1/2}(I_n+6c_F n),
	\end{equation}
	where $I_n$ is defined by \eqref{radon_functional}, $\Gamma$-converges with respect to the weak* convergence of measures to the functional $I_\infty$ defined by
		\begin{equation}
	I_\infty(\mu):=\begin{cases}  \mathcal{E}(D_\mu), &  \text{if \FFF there exists \EEE $D_\mu\subset\Rz^2\setminus S$  set of finite perimeter}\\
	& \hspace{15ex} \text{with  $|D_\mu|=1/\rho$ such that $\mu=\rho\chi_{D_\mu}$,}
\\
	+\infty, &\text{otherwise,}\\
	\end{cases}
	\end{equation}
	for every $\mu\in\mathcal{M}(\Rz^2)$,
where $\rho:=2/\sqrt{3}$.
	\item[2.]  The functional $I_\infty$ admits a minimizer in 
	\begin{align}
\mathcal{M}_W:=\bigg\{\mu\in\mathcal{M}(\Rz^2)\ :\ \text{$\exists$ $D\subset\Rz^2\setminus S$} & \hspace{2ex} \text{set of finite  perimeter, bounded }\notag\\
\hspace{8ex} \text{with }  |D|=\frac{1}{\rho},\label{M_Wulff}	& \hspace{2ex} \text{ and such that $\mu=\rho\chi_D$}\bigg\}. 
	\end{align}
	\item[3.]  Every sequence $\mu_n\in\mathcal{M}_n$ of minimizers of $E_n$ 
admits, up to translation in the direction $\tone$ $($i.e., up to replacing $\mu_n$ with $\mu_n(\cdot+c_n\tone)$ for  chosen fixed integers $c_n\in\Zz$$)$, 
a subsequence  converging  with respect to the weak* convergence of measures to a minimizer of $I_\infty$ in $\mathcal{M}_W$. 
	\end{itemize}
	\end{theorem}

\PPP
\noindent Notice that the parameter $\rho:=2/\sqrt{3}$ in the definition of $\mathcal{M}_W$ is related to the fact that we chose the triangular lattice for $\mathcal{L}_F$, as $\rho$ is the density of atoms per unit volume of such lattice. \EEE
 \EEE

\section{\PPP Wetting regime}\label{sec:wetting}

In this section we \PPP single out \EEE conditions that entail  \emph{wetting}, i.e., the situation in which it is more convenient for film atoms to spread on the infinite substrate surface instead of accumulating in clusters, or islands, on top of it. \PPP In the following we refer to crystalline configurations $D^{\rm w}_n\subset\partial\mathcal{L}_{FS}$ as \emph{wetting configurations}. We first consider  the case  $q\neq 1$.  \EEE

 \begin{proposition}\label{wetting_regime1}
Let $\MMM q\neq 1$ \PPP and  $n \in \mathbb{N}$.  Any wetting configuration $D^{\rm w}_n:=\{w_1,\dots,w_n\}\subset\partial\mathcal{L}_{FS}$ \EEE 
satisfies \PPP the following two assertions: \EEE
\begin{itemize} 
\item[(i)] $V_n(D^{\rm w}_n)=\min{V_n(D_n)}$\PPP, \EEE
\item[(ii)] \PPP $V_n(D^{\rm w}_n)<V_n(D_n)$ \EEE for  \PPP any crystalline configuration $D_n$ with $D_n\setminus\partial\mathcal{L}_{FS}\neq\emptyset$,  
\end{itemize}
 if and only if 
\begin{equation}\label{wetting_condition1}
c_{S} \geq 6 c_{F}.
\end{equation}
  \end{proposition}

\begin{proof}

\PPP We begin by proving the sufficiency of \eqref{wetting_condition1} for the assertions (i) and (ii). Note that (i) easily follows from (ii) and the fact that any wetting configuration $D^{\rm w}_n$  has the same energy given by 
\begin{equation}\label{wetting_energy}
V_n (D^{\rm w}_n)=-c_S n.
\end{equation}
In order to prove (ii) we proceed by induction on $n$. We first notice that (ii) is trivial for $n=1$. Then, we  assume that (ii) holds true for every $k=1,\dots,n-1$ and prove that it holds also for $n$. Let $D_n$ be a crystalline configuration such that $D_n\setminus\partial\mathcal{L}_{FS}\neq\emptyset$. If $D_n\cap(\mathbb{R}\times\{r>e_S \})=\emptyset$, we can easily see that the energy of $D_n$ is higher of the energy of $D^{\rm w}_n$ at least by $c_S-2c_F$, which is positive by  \eqref{wetting_condition1},  since the elements in $D_n\setminus\partial\mathcal{L}_{FS}\neq\emptyset$ have at most two film bonds and no substrate bonds. 
Therefore, we can assume that  $D_n\cap(\mathbb{R}\times\{r>e_S \})\neq\emptyset$.  Let $L$ be the last line in $\Rz\times\{r>0\}$ parallel to $\tone$ that intersects $D_n$ by moving upwards from $\Rz\times\{e_S\}$ (which exists since $D_n$ has a finite number of atoms). 

We claim that 
\begin{equation}\label{wetting_induction}
 V_n(D_n) \geq  V_{n-\ell}(D_n\backslash L)-6c_F(\ell-1)-4c_F,
\end{equation}
where $\ell:=\#(D_n \cap L)$. We order the element of $D_n \cap L$ with increasing indexes with respect to $\tone$, i.e., $D_n \cap L=\{x_1,\dots,x_\ell\}$, and observe that $x_1$ has at most 3 bonds with film atoms in $D_n$ by construction, since $x_1$ is the leftmost element in $D_n \cap L$.   We notice that in the same way, if $\ell>1$, every $x_i$ has at most 3 bonds with film atoms in $D_n\setminus\{x_1,\dots,x_{i-1}\}$ for every $i=2,\dots,\ell-1$. Therefore,  we obtain that
\begin{align*}
 V_n(D_n) &\geq  V_{n-1}(D_{n-1} \backslash \{x_1\})-6c_F\geq  V_{n-i}(D_{n-i}\backslash \{x_1,\dots,x_{i}\})-6c_Fi\\
 & \geq V_{n-(\ell-1)}(D_n\backslash \{x_1,\dots,x_{\ell-1}\})-6c_F(\ell-1) \geq  V_{n-\ell}(D_n\backslash L)-6c_F(\ell-1)-4c_F,
\end{align*}
which in turns is \eqref{wetting_induction}, where in the last inequality we used that $x_\ell$ has only at most 2 bonds with film atoms in $D_n\setminus L$, since $x_\ell$ is the rightmost element in $D_n \cap L$. 

From \eqref{wetting_induction} it follows that 
\begin{align*}
 V_n(D_n) & \geq  V_{n-\ell}(D_n\backslash L)-6c_F(\ell-1)-4c_F> V_{n-\ell}(D_n\backslash L)-6c_F\ell\\
 &> -c_S (n-\ell)-6c_F\ell\geq -c_Sn,
\end{align*}
where we used the induction and \eqref{wetting_energy}  in the third inequality, and \eqref{wetting_condition1} in the last inequality. 

 To prove the necessity of \eqref{wetting_condition1} notice that \MMM the Wulff configuration \PPP in $\Rz\times\{r>e_s\}$ has energy equal to $-6c_Fn+C\sqrt{n}$ \PPP for some constant \EEE $C>0$. Therefore, from assertion (ii) and \eqref{wetting_energy} it follows
$$\PPP-c_S n  <\EEE -6c_Fn +C \sqrt{n}. $$ 
After dividing by $n$ and letting $n \to \infty$ we obtain
$c_S \geq  6 c_F$.


\end{proof}

\FFF \begin{remark} 
Notice from the proof of Proposition \ref{wetting_regime1} that for the necessity of \eqref{wetting_condition1} it is enough assertion (i) or,  more precisely, it is enough that 
there exists an increasing subsequence $(n_k)_{k \in \mathbf{N}}$ such that (i) holds for every $n_k$. 		
\end{remark} \EEE

\PPP We now address the case $q=1$ for which we notice that $\partial\mathcal{L}_{FS}=\partial\mathcal{L}_{F}$.\EEE


 \begin{proposition}\label{wetting_regime2}
Let \PPP $q= 1$ and $n \in \mathbb{N}$. Any configuration  $D^{\rm w}_n:=\{w_1,\dots,w_n\}\subset\partial\mathcal{L}_{FS}$  such that 
\begin{equation}\label{connected_wetting}
w_{i+1}:=w_i+\tone
\end{equation}
for every $i=1,\dots,{n}$,  satisfies \PPP the following two assertions: \EEE
\begin{itemize} 
	\item[(i)] $V_n(D^{\rm w}_n)=\min{V_n(D_n)}$,
	\item[(ii)] \PPP $V_n(D^{\rm w}_n)<V_n(D_n)$ \EEE for  \PPP any crystalline configuration $D_n$ such that either $D_n\setminus\partial\mathcal{L}_{FS}\neq\emptyset$ or not satisfying \eqref{connected_wetting}, 

\end{itemize}
if and only if  
\begin{equation}\label{wetting_condition2}
c_{S}\geq 4 c_{F}.
\end{equation}

\end{proposition} 
\begin{proof}
\PPP The proof is based on the same arguments employed for Proposition \ref{wetting_regime1} and on the following observations. Any wetting configuration $D^{\rm w}_n$ satisfying \eqref{connected_wetting} has the same energy given by 
\begin{equation}\label{wetting_energy2}
V_n (D^{\rm w}_n)=-c_S n -2c_F(n-1).
\end{equation}

In order to prove the sufficiency of \eqref{wetting_condition2} for assertion (ii) (assertion (i) follows in view of \eqref{wetting_energy2}), we can restrict also in this case without loss of generality to configurations $D_n\cap(\mathbb{R}\times\{r>e_S \})\neq\emptyset$, since any wetting configuration that does not satisfy \eqref{connected_wetting} has energy obviously higher than \eqref{wetting_energy2} (because $n-1$ is the maximum number of bonds in $\partial\mathcal{L}_{F}$). 

In order to prove the necessity of \eqref{wetting_condition2} for assertions (i) and (ii), we again consider the Wulff shape with $n$ atoms in $\Rz\times\{r>e_s\}$ which has energy $-6c_Fn+C\sqrt{n}$ for some constant $C>0$, and observe that  
$$-c_S n -2c_F(n-1) < -6c_Fn +C \sqrt{n} $$ 
by assertion (ii) and \eqref{wetting_energy2}.

\end{proof}

\FFF \begin{remark} 
Notice from the proof of Proposition \ref{wetting_regime2} that for the necessity of \eqref{wetting_condition2} it is enough assertion (i) or,  more precisely, it is enough that 
there exists an increasing subsequence $(n_k)_{k \in \mathbf{N}}$ such that (i) holds for every $n_k$. 		
\end{remark} \EEE	

\PPP We refer to \eqref{wetting_condition1} and \eqref{wetting_condition2} as \emph{wetting conditions}.  \EEE
\PPP Condition \eqref{wetting_condition2} is weaker than \eqref{wetting_condition1} because if $q=1$, then film atoms of wetting configurations can be bonded to the two film atoms at their sides in $\partial\mathcal{L}_{FS}$ (if filled) besides to their corresponding substrate atom, and Proposition \ref{wetting_regime2} show that such configuration are preferable. 
We notice that the same arguments used in Propositions \ref{wetting_regime1}  and \ref{wetting_regime2} work for other rigid positioning of $\mathcal{L}_F$ and  $\mathcal{L}_S$. For example for the case with  
$$
e_S:=\frac{3}{4}\qquad\text{ and}\qquad x_F:=\left(-\frac{1}{8},\frac{1}{8}\sqrt{35}\right) 
$$
 the wetting condition  \eqref{wetting_condition1} is replaced by $c_S \geq 5c_F$, since film atoms in wetting configurations may present a film bond besides the substrate bond. 

\begin{proof}[Proof of Theorem \ref{wetting_theorem}]
The assertion directly follows from Proposition  \ref{wetting_regime1} and Proposition  \ref{wetting_regime2} for the case $q\neq1$ and the case $q=1$, respectively.
\end{proof} \EEE

\section{Compactness}\label{sec:compactness}

\PPP In the remaining part of the paper we work in the dewetting regime, i.e., under the assumption  \eqref{dewetting_condition}. \EEE
We begin by \PPP establishing a lower bound in terms of $c_F$ and $c_S$ \PPP of the \EEE strip energy $E_{\rm strip}(x)$ uniform for \PPP every \EEE $x\in D_n\cap\partial \mathcal{L}_F$. To this aim, \PPP we need to distinguish the case $q= 1$ from $q\neq1$ as already done in Section \ref{sec:wetting} because of the different contributions in $E_{\rm strip}(x)$ of the substrate interactions. \EEE

 \begin{lemma}\label{strip_local}
 We have that 
$$
E_{\rm strip}(x)\geq 
\Delta_{\rm strip}
$$
with

\begin{equation}\label{delta_strip}
\Delta_{\rm strip}:=\begin{cases} 6c_F-c_S, &\text{if $\MMM q\neq 1 \EEE$,}\\
4c_F-c_S, &\text{if $\MMM q= 1 \EEE$},
\end{cases}
\end{equation}

for every $x\in D_n\cap\partial \mathcal{L}_F$.
 \end{lemma}
   \begin{proof}
\FFF Fix \EEE $x\in D_n\cap\partial \mathcal{L}_F$. 
We begin by observing that the strip center $x$ surely misses the bonds with the atoms missing at the 2 positions $x-\ttwo+k\tone$ for $k=0,1$ \FFF as shown in Figure \ref{fig:strip}\EEE. Furthermore, either $x$ misses the bond with $x_{-}$ or $x_{-}\in {\color{black} D_n} $  and $x_{-}$ misses the bonds with the 2 positions $x-\ttwo+k\tone$ for $k=-1,0$ (which in the strip energy are counted with half weights). We can reason similarly for $x_{+}$. Therefore, by the definition of energy of the low strip $E_{\rm strip,below}$, 
$$
E_{\rm strip,below}\FFF \geq\EEE\begin{cases}
 4c_F-c_S, &\text{if $q\neq1$,}\\
 2c_F-c_S, &\text{if $q=1$,}
\end{cases}
$$

We analyze $E_{strip, above}$. There are several possibilities:
	\begin{enumerate} 
	\item neither of $\tilde{x}_+$ and $\tilde{x}_-$ belongs to $D_n$; 
	\item  exactly one of $\tilde{x}_+$ and $\tilde{x}_-$ belongs to $D_n$; 
	\item both  $\tilde{x}_+$ and $\tilde{x}_-$ belong to $D_n$. 		
\end{enumerate} 		
In case of  (1) we have the contribution of $2c_F$ since $\tilde x$ misses two bonds. In case of (3) each of $\tilde{x}_+$ and $\tilde{x}_-$ misses at least one bond (namely with $\tilde{x}+2\ttwo-\tone$ which is not in $D_n$ due to the definition of $\tilde x$). If $\tilde{x}_{\pm}\neq \widetilde{(x_{\pm})}_{\mp}$ we have the energy contribution of at least $2c_F$. On the other hand if it is valid that  $\tilde{x}_{\pm}= \widetilde{(x_{\pm})}_{\mp}$, we have the energy contribution of $c_F$ due to the missing bond with $\tilde{x}+2\ttwo-\tone$ and each of $\tilde{x}_{\pm}$ misses one more bond (namely with $\tilde{x}+2\ttwo$ and $\tilde{x}+2\ttwo-2\tone$, which in this case do not belong to $D_n$). The similar analysis can be made if $\tilde{x}_{+}= \widetilde{(x_{+})}_{-}$ or $\tilde{x}_{-}= \widetilde{(x_{-})}_{+}$. 
Thus we have again energy defficiency of $2c_F$. 
Finally in the case of (ii) 
without loss of generality we assume that $\tilde x_{+} \in D_n$. 
$\tilde x$ is already missing one bond (one $c_F$) and again one bond of $\tilde x_{+}$ is missing since $\tilde{x}+2\ttwo-\tone$ is not in $D_n$. Again, this bond is counted as one $c_F$, if $\tilde{x}_{+}\neq  \widetilde{(x_{+})}_{-}$ and as $c_F/2$, if $\tilde{x}_{+}\neq  \widetilde{(x_{+})}_{-}$. In this case one more  $c_F/2$ we obtain since $\tilde{x}_{+}$ is missing one bond with $\tilde{x}+2\ttwo$. 

Therefore, in the strip energy $E_{\rm strip}$ the terms related to the triple $\tilde{x}$, $\tilde{x}_+$, and $\tilde{x}_-$ give a contribution of at least $2c_F$.

   \end{proof}

We  now observe that the energy $V_n(D_n)$ of any crystalline configuration $D_n$ is bounded below by  $-6c_Fn$ plus a positive \emph{deficit} due to the boundary of $D_n$ where atoms have less than 6 film bonds \PPP and \EEE could have a bond with the substrate.  

 \begin{lemma}\label{strip}
 If \eqref{dewetting_condition} holds, then there exists $\Delta>0$ such that
\begin{equation}\label{Vn_lower_bound}
 V_n(D_n)\geq -6c_Fn\,+\,\Delta\#\partial D_n
\end{equation}
 for every crystalline configuration $D_n\subset\mathcal{L}_F$. \PPP Furthermore, the following two assertions are equivalent:
 \begin{itemize}
  \item[(i)] There exists a constant $C>0$ such that $\# \partial D_n \leq C \sqrt n$ for every $n \in \mathbb{N}$,
 \item[(ii)]  There exists a constant $C'>0$ such that $E_n (\mu_{D_n}) \leq C'$ for every $n \in \mathbb{N}$.
 \end{itemize}

  \end{lemma}

 \begin{proof}
We begin by observing that from \eqref{Eloc} and \eqref{Estrip} it follows that
\begin{align}
6c_Fn\,+\,
V_n(D_n)\,&=\, \sum_{x\in D_n}\left(\sum_{y\in {\color{black} D_n}\setminus\{x\}} v_{FF}(|x-y|)+ 6c_F
\right)\,+\,\sum_{x\in D_n}  v^1(x)\notag\\
&=\, \sum_{x\in D_n} E_{\rm loc}(x)\,+\,\sum_{x\in D_n}  v^1(x)\notag\\
&\geq \, \sum_{x\in D_n\cap {\color{black} \partial \mathcal{L}_{FS}}} E_{\rm strip}(x)\,+\,\sum_{x\in D_n\setminus\mathcal{S}({\color{black} \partial \mathcal{L}_{FS}})} E_{\rm loc}(x),\label{energy_strip}
\end{align}
where
\begin{equation}\label{layer_strip}
\mathcal{S}(\partial \mathcal{L}_{FS})=\mathcal{S}_{\MMM D_n\EEE}({\color{black} \partial \mathcal{L}_{FS}}):=\{y\in\mathcal{S}(x)\,:\,x\in D_n\cap{\color{black} \partial \mathcal{L}_{FS}} \},
\end{equation}
because  $v^1(x)=0$ for every $x\in D_n\setminus\partial \mathcal{L}_{FS}$ and the careful choice of the weights in \FFF \eqref{Estrip_below1},  \eqref{Estrip_below2}, and \eqref{Estrip_above} with \eqref{weigths_above}. \EEE More precisely, we notice that for every point in $D_n\cap {\color{black} \partial \mathcal{L}_{FS}}$ the local energy $E_{\rm loc}(x)$ is counted {\color{black} at most once}. The weights $c_{\pm}(\tilde{x}_{\pm})$ are instead chosen so that the local energy of $\tilde{x}_{\pm}$ is fully counted if $\tilde{x}_{\pm}$ do not belong to the next strip and only half in the other case. {\color{black} Thus these weights are also at most one}. 
We now observe that 
\begin{equation}\label{lower_bound_1}
\sum_{x\in D_n\setminus\mathcal{S}(\FFF \partial  \mathcal{L}_{FS}\EEE)} E_{\rm loc}(x)\geq 
c_F \#(\partial D_n\setminus\mathcal{S}({\color{black} \partial \mathcal{L}_{FS}}))
\end{equation}
because every point in $D_n\setminus\mathcal{S}(\partial \mathcal{L}_{FS})$ has 6 bonds if not on the $\partial D_n$ where at least one bond is missing by definition.

Therefore, by \eqref{energy_strip}, \eqref{lower_bound_1}, and Lemma  \ref{strip_local}  we obtain that
\begin{align}
6c_F n\,+\, V_n(D_n)\,&\geq \, \sum_{x\in D_n\cap{\color{black} \partial \mathcal{L}_{FS}}} E_{\rm strip}(x)\,+\,\sum_{x\in D_n\setminus\mathcal{S}({\color{black} \partial \mathcal{L}_{FS}} )} E_{\rm loc}(x),\notag\\
&\geq \Delta_{strip}  \#(D_n\cap {\color{black} \partial \mathcal{L}_{FS}}) \,+\, 
c_F \#(\partial D_n\setminus\mathcal{S}({\color{black} \partial \mathcal{L}_{FS}}))\notag\\
&\geq \min\left\{\frac{\Delta_{strip}}{6}, c_F\right\}\#\partial D_n\label{lower_bound_3}
\end{align}
where in the last inequality we used that $\MMM\#\mathcal{S}(\partial \mathcal{L}_{FS})\leq6\#D_n \cap \partial\mathcal{L}_{FS}\EEE$. The assertion now easily follows from \eqref{lower_bound_3} by choosing  
$$
\Delta:=\min\left\{\frac{\Delta_{strip}}{6}, c_F\right\}>0,
$$
\MMM where we used \eqref{dewetting_condition}. \EEE

\PPP To prove the last assertion we observe that assertion (i) implies (ii) since by \eqref{energy_equivalence} and \eqref{converging_energies00} 
$$
\sqrt{n}E_n(\mu_{D_n})=V_n(D_n)+6c_Fn\leq 6c_F\#\partial D_n,
$$
where in the last equality we used the definition of $\partial D_n$. Furthermore, also by \eqref{Vn_lower_bound},
$$
\Delta\#\partial D_n\leq V_n(D_n)+6c_Fn=\sqrt{n}E_n(\mu_{D_n})
$$
and hence, assertion (ii) implies (i). \EEE

\EEE
 \end{proof}

\PPP In view of the previous lower bound for the energy of a configuration $D_n$ we are now able to prove a compactness results.  \PPP We notice that to achieve compactness the negative contribution coming at the boundary from the interaction with the substrate needs to be compensated. This is not trivial, e.g., in the case $6c_F>c_S>4c_F$, where atoms $x$ of configurations on $\partial\mathcal{L}_{FS}$ have one bond with a substrate atom and at least two bonds with film atoms  missing. A way to solve the issue is to look for extra positive contributions from other atoms in the boundary. However, just looking for neighboring atoms might be not enough, e.g., in the case with $e_S=2$ or  $e_S=\frac{2}{3}$.  The issue is solved in the proof of the following compactness result by introducing a new non-local argument called the \emph{strip argument} that involve looking at the whole strip $\mathcal{S}(x)$. The same argument would work  for other rigid positioning of $\mathcal{L}_F$ and  $\mathcal{L}_S$, such as for  
$$
e_S:=\frac{3}{4}\qquad\text{ and}\qquad x_F:=\left(-\frac{1}{8},\frac{1}{8}\sqrt{35}\right).
$$
\EEE
 

\PPP We conclude the section with compactness results for sequences of almost-connected configuration (see Section \ref{sec:almost_connected} for the definition). We remind the reader that by the trasformation defined in Definition \ref{transformation} for any configuration $D_n$ there exists the almost-connected configuration $\widetilde{D}_n$ such that $V_n(\widetilde{D}_n)\leq V_n(D_n)$. \EEE

 \begin{proposition} \label{zadnje} 
\PPP Assume that \eqref{dewetting_condition} holds. Let  $D_n\in\mathcal{C}_n$ be almost-connected configurations such that
\begin{equation}\label{upperbound} 
V_n(D_n)\leq -6c_Fn +Cn^{1/2}
\end{equation}
for a constant $C>0$. Then there exist an increasing sequence $n_r$, $r\in\Nz$, and a measure $\mu\in \mathcal{M}(\Rz^2)$ with $\mu\geq0$ and $\mu(\Rz^2)=1$  such that   $\mu_{r}\rightharpoonup^*\mu$ in $\mathcal{M}(\Rz^2)$, where $\mu_{r}:=\mu_{D_{n_r}\PPP(\,\cdot\,+a_{n_r})}$ for some translations $a_n\in\Rz^2$ \EEE \emph{(}see  \eqref{empiricalmeasures} for the definition of  the empirical measures $\mu_{D_{n_r}}$\emph{)}. \PPP Moreover, if $D_n\in\mathcal{C}_n$ are minimizers of $V_n$ in $\mathcal{C}_n$, then we can choose $a_n=t_n\tone$ for integers $t_n\in\Zz$. \EEE
 \end{proposition} 
 
 \begin{proof} \FFF In the following we denote by $B(x,R)$ an open ball of radius $R>0$ centered at $x \in \mathbb{R}^2$ and we define $B(R):=B(o,R)$ where $o$ is the origin in $\mathbb{R}^2$. \EEE 
We want to show that \PPP there exists \FFF $R>0$ such that $D_{n}\subset B(R)$ (up to a translation) for every $n$. \EEE 

\FFF To this aim \EEE  we denote for any $D_n$ \EEE its  \FFF $k:=k_{D_n}$ \EEE connected components \FFF by \EEE $D_n^i$  for $\FFF i\EEE=1,\dots,k$.
We define  the sets 
$$\FFF \Omega_i:= \bigcup_{x\in    D_n^i} \nu_{\textrm{trunc}}(x),$$
for $i=1,\dots,k$, where $$\nu_{\textrm{trunc}}(x):=\nu(x)\cap  B(x,q)$$ 
\FFF with $q$ defined in \eqref{eSratio} and $\nu(x)$ denoting  \EEE  the \FFF (closed) \EEE Voronoi cell associated to $x$ \FFF with respect to  $D_n^{i}$, i.e., 
\begin{equation}\label{voronoi}
\nu(x):= \{ y\in\Rz^2\ :\ \text{$|y-x|\leq|y-x'|$ for all $x'\in D_n^{i}\setminus\{x\}$}\}, \EEE
\end{equation}
\EEE
and we observe that by construction \FFF and  the convexity of $\nu(x)$\EEE,
\begin{equation}\label{voronoibound} 
|\partial \nu_{\textrm{trunc}}(x)|\leq 2\MMM q \pi.
\end{equation}
We claim that \FFF $\Omega_i$ are \EEE connected. \FFF Indeed,  \EEE
if $x,y \in  D_n$ are such that  $|x-y|=1$,  then it is easily seen that the midpoint on the line that connects $x$ and $y$ belongs to both $\nu(x)  \cap  \nu(y)$ and $B (x,q) \cap B (y,q)$. The second claim \FFF easily follows from \eqref{eSratio} \EEE while the first claim follows from \PPP the \EEE triangular inequality (it is impossible that for \PPP every \EEE $z \in    D_n$ it is valid 
$$\left|z- \frac{x+y}{2}\right|<1/2$$ since then by \FFF the \EEE triangular inequality $z$ would be \PPP distant from  both $x$ and  $y$ less \EEE than one). 

We now claim that also
$$\Omega:= \bigcup_{x\in  D_{n}} \nu_{\textrm{trunc}}(x), $$
is connected. This follows by showing that $\Omega_i$ and \FFF$\cup_{i=1}^{i-1}\EEE\Omega_l$ are connected for $i=2, \dots, k$, which  in turns is a consequence of the fact that by definition \EEE $D_n^i$  is separated by at most $q$ from \FFF$\cup_{l=1}^{i-1}\EEE D_n^l$ for $i=2,\dots,k$. In fact, by \EEE  the same reasoning used in the previous claim applied this time to two points $x \in D_n^i$ and $y \in \cup_{l=1}^{i-1}D_n^l$ chosen such that $|x-y|= \textrm{dist}(D_n^i,\cup_{l=1}^{i-1} D_n^l)$, \FFF where $\textrm{dist}(A,B)$ with respect to two subsets $A$ and $B$ of $\mathbb{R}^2$ denotes  the distance between them, we can deduce that   $(x+y)/2$ belongs to both $\nu(x)  \cap  \nu(y)$ and $ B(x,q) \cap B (y,q)$, which yields the claim. \EEE


Therefore, we have that
\begin{equation}\label{diameter_bound} 
\textrm{diam}(D_{n}):=\max_{x,y\in D_{n}} |x-y|\leq \textrm{diam}(\Omega)\leq \frac{1}{2} |\partial \Omega|\leq  \frac{1}{2} \sum_{x\in \partial D_n}|\partial \nu_{\textrm{trunc}}(x)| \leq \pi \MMM q \#\partial D_{n}
\end{equation}
where \FFF $\textrm{diam} (A)$ of a set $A$ is the diameter of $A$ and we used that $\Omega$ is connected \FFF in the second inequality, that if $x \in    D_n$ has $6$ film neighbors, then by elementary geometric observation $\nu_{\textrm{trunc}} (x) \cap \partial \Omega=\emptyset$ in the third inequality,  \EEE and \eqref{voronoibound} in the last inequality\EEE. 

Finally, from   \eqref{upperbound}, \eqref{diameter_bound} and  Lemma \ref{strip} we obtain that 
$$
\textrm{diam}(D_{n})\leq  \frac{C\pi \MMM q}{\Delta} n^{1/2}
$$
\FFF and hence, by \eqref{empiricalmeasures} there exist translations $\mu_n$  of $\mu_{D_n}$ such that $\textrm{supp} \mu_n\subset  B(R)$ for some  $R>C\pi q/2\Delta$ and for every $n$. 
Therefore, \EEE since $|\mu_{D_n}|(\Rz^2)=1$ for every $n$,  by \cite[Theorem 1.59]{AFP} there exist a subsequence $(n_r)_{r} \in \mathbf{N}$ 
and a measure $\mu\in \mathcal{M}(\Rz^2)$ such that \PPP $\mu_{r}\stackrel{*}{\rightharpoonup}\mu$ in $\mathcal{M}(\Rz^2)$. Furthermore, $\mu\geq0$ and
$$
\mu(\Rz^2)\leq\lim_{r\to\infty}  {\mu}_r(\Rz^2)=1.
$$
In order to conclude the proof  it suffices to prove that $\mu(\Rz^2)=1$, and this directly  follows from the fact that  \FFF the support of \EEE $\mu_r$ are contained \FFF in a compact set of $\Rz^2$.  \EEE
\end{proof}

\PPP  The following compactness result is the analogous of \cite[Theorem 1.1]{Yeung-et-al12} in our setting with substrate interactions. 

\EEE

 \begin{theorem}[Compactness] \label{compactnesstheorem}
 Assume  \eqref{dewetting_condition}.  Let $D_n\in\mathcal{C}_n$ be configurations  satisfying \eqref{upperbound} and \PPP let \EEE $\mu_{n}:=\mu_{\mathcal{T}(D_n)}$ \PPP be \EEE the empirical measures associated to the transformed configurations \PPP $\mathcal{T}(D_n)\in\mathcal{C}_n$ associated to $D_n$  by Definition \ref{transformation}. \EEE  Then, \PPP up to translations $($i.e., up to replacing $\mu_{n}$ by  $\mu_{n}(\cdot+a_n)$ for some $a_n\in\Rz^2$$)$ and a passage to a non-relabelled subsequence,  $\mu_{n}$ converges weak*  in $\mathcal{M}(\mathbb{R}^2)$ to a measure $\mu\in\mathcal{M}_W$, where	$\mathcal{M}_W$ is defined in \eqref{M_Wulff}. \PPP Furthermore, if $D_n\in\mathcal{C}_n$ are minimizers of $V_n$ in $\mathcal{C}_n$, then we can choose $a_n=t_n\tone$ for integers $t_n\in\Zz$. 
\end{theorem} 

\PPP
\begin{proof}
We begin by observing that the transformed configurations $\mathcal{T}(D_n)$ of the configurations $D_n$ are almost-connected configurations  in $\mathcal{C}_n$ since they result from applying transformation $\mathcal{T}_2$, and that 
\begin{equation}\label{energy_lower}
V_n(\mathcal{T}(D_n))\leq V_n(D_n),
\end{equation}
since no active bond of $D_n$ is deactivated by performing the transformations $\mathcal{T}_1$ and $\mathcal{T}_2$ (see Definition \ref{transformation} for the definition of $\mathcal{T}_1$ and $\mathcal{T}_2$). Therefore, in view of Proposition \ref{zadnje} by  \eqref{upperbound} and  \eqref{energy_lower} we obtain that, up to a non-relabeled subsequence, there exist $a_n\in\Rz^2$ and a measure $\mu\in \mathcal{M}(\Rz^2)$ with $\mu\geq0$ and $\mu(\Rz^2)=1$  such that   $$\mu_{\mathcal{T}(D_n)}(\cdot+a_n)\rightharpoonup^*\mu$$ in $\mathcal{M}(\Rz^2).$
We can then conclude that  $\mu\in\mathcal{M}_W$ by directly applying the arguments in the proof of \cite[Theorem 1.1]{Yeung-et-al12}.

\end{proof}

\PPP
 We notice that, if the sequence $D_n\in\mathcal{C}_n$ is a sequence of almost-connected configurations, then Theorem \ref{compactnesstheorem} directly  holds for $D_n$ without  the need to pass to the associated transformed configurations  $\mathcal{T}(D_n)$ given by Definition \ref{transformation}.
\EEE

\section{Lower bound}\label{sec:lower_bound}

\PPP

We denote by $h_{1/\sqrt{3}}(x)$ the interior part of the Voronoi cell associated to every  $x\in\mathcal{L}_F$ with respect to $\mathcal{L}_F$, i.e., 
$$
h_{1/\sqrt{3}}(x):= \left\{ y\in\Rz^2\ :\ \text{$|y-x|<|y-x'|$ for all $x'\in \mathcal{L}_F\setminus\{x\}$}\right\}
$$
that is an open hexagon of radius $1/\sqrt{3}$, and by $v(x)$ its scaling in $\mathcal{L}_F/\sqrt{n}$, i.e.
\begin{equation}\label{scaled_voronoi}
v(x):= \frac{h_{1/\sqrt{3}}(x)}{\sqrt{n}}.
\end{equation}
 Given a configuration $D_n$, we consider the auxiliary set $H_n$  associated to $D_n$ which was introduced in  \cite{Yeung-et-al12} and defined by 
\begin{equation}\label{H1}
H_n= \bigcup_{x \in D_n} \overline{v(x)}.
\end{equation} 
The boundary of $H_n$ is given by the union of  a number $M\in\Nz$ (depending on $D_n$) of closed polygonal boundaries $P_1,\dots,P_M$. For $k=1,\dots,M$ we denote the $m_k\in\Nz$ vertices of $P_k$ by $v_1^k,\dots,v_{m_k}^k$ and we set $v_{m_k+1}^k:=v_1^k$, so that
$$
P_k:=\bigcup_{i=1}^{m_k}[v_{i+1}^k,v_i^k]
$$
where $[a,b]$ denotes the closed segment with endpoints $a,b\in\Rz^2$. 
Notice that each $m_k$ is even and that we can always order the vertices so that
$$
v_{2i}^k\in V_{\mathcal{L}_F}^{\rm e}:=\left( \frac{1}{3\sqrt{n}}(\tone+\ttwo)+\frac{1}{\sqrt{n}}\mathcal{L}_F\right)
$$
and 
$$
v_{2i-1}^k\in 
V_{\mathcal{L}_F}^{\rm o}:=\left(\frac{1}{3\sqrt{n}}(2\tone-\ttwo)+\frac{1}{\sqrt{n}}\mathcal{L}_F\right)
$$
(see Figure \ref{fig:auxiliary}). To avoid the atomic-scale oscillations in $\partial H_n$ between the two sets of vertices $V_{\mathcal{L}_F}^{\rm e}$ and $V_{\mathcal{L}_F}^{\rm o}$, we introduce another auxiliary set denoted by $H'_n$ where such oscillations are removed, by considering only the vertices in one of the two sets, say  $V_{\mathcal{L}_F}^{\rm o}$ as depicted in Figure \ref{fig:auxiliary}. More precisely, the set $H'_n\subset\Rz^2$ is defined 
as the unique set with $D_n\subset H_n'$ such that 
\begin{equation}\label{H2}
\partial H_n':=\bigcup_{k=1}^M P_k',
\end{equation}
where
$$
P_k':=\bigcup_{i=1}^{m_k/2}[v_{2i-1}^k,v_{2i+1}^k].
$$
\begin{figure}
\includegraphics[width=0.9\textwidth]{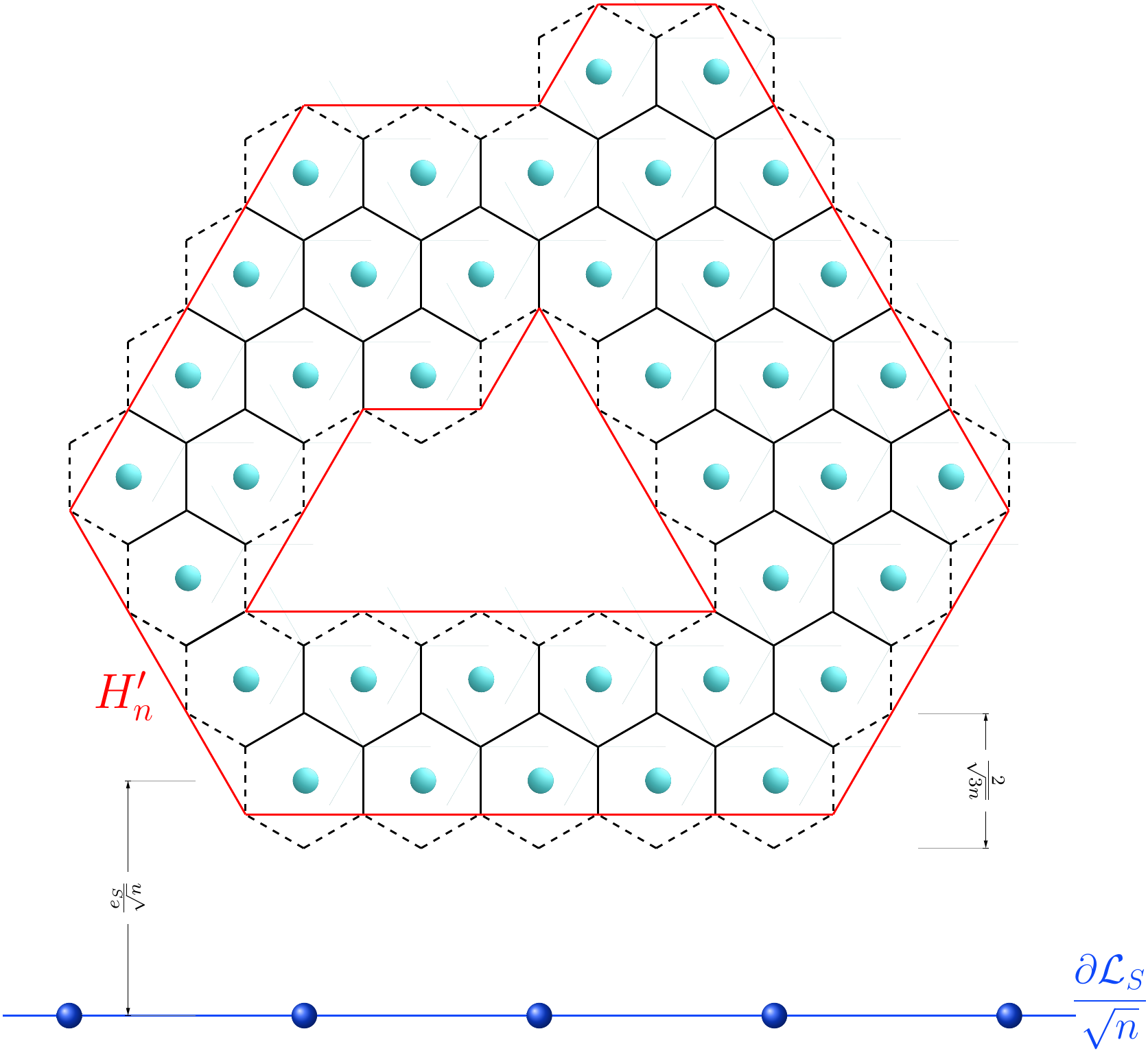}
\caption{\PPP A configuration $D_n$ is depicted with scaled Voronoi cells  $v(x)$ for every $x\in D_n$. The boundary of $H_n$, which in this example consists of two polygonal lines (one \FFF ``internal'' and one ``external''), \EEE is indicated with a dashed black line while the boundary of $H_n'$ with a continuous red line.}
\label{fig:auxiliary}
\end{figure}

\EEE

\PPP It easily follows from the construction of the auxiliary sets $H_n$ and $H_n'$ associated to the configuration $D_n$ that \EEE
\begin{equation} \label{comparison1} 
 |H_n \Delta H_n'| \leq \frac{\# \partial D_n}{8\PPP n\EEE \sqrt 3},  
  \end{equation} 
 and
 \begin{equation} \label{comparison2} 
 \left|\PPP\mathcal{H}^1(\partial H_n)\EEE-\PPP\mathcal{H}^1(\partial H_n')\EEE\right| \leq 2\sqrt{3} \frac{\# \partial D_n}{\sqrt n}.  
 \end{equation}

In the following we use the notation 
$$\partial\mathcal{L}_{FS}^n:= \frac{\partial\mathcal{L}_{FS}}{\sqrt{n}}.$$ 
For every point $y\in\partial\mathcal{L}_{FS}^n$ we denote \PPP its left and right half-open intervals with length $1/\sqrt{n}$ by 
$$\PPP I^{+}_y:=\left[y,y+\frac{1}{2\sqrt{n}}\right)\quad\textrm{and}\quad I^{-}_y:=\left(y-\frac{1}{2\sqrt{n}}, y\right],$$
respectively, and we define the \emph{oscillatory sets of $y\in\partial\mathcal{L}_{FS}^n$} by
$$
O_n^{y}:=O_n^{y,-}\cup O_n^{y,+},
$$
where $O_n^{y,\pm}$ are the \emph{left and right oscillatory sets of $y\in\partial\mathcal{L}_{FS}^n$}, i.e., 
$$\PPP O_n^{y,\pm}:=\{I^{\pm}_y\times \Rz: y\in\partial\mathcal{L}_{FS}^n\}$$
(see Figure \ref{fig:lattices}).
The (overall) \emph{oscillatory set} $O_n$ is defined as 
\begin{equation} \label{total_oscillation} 
O_n:=\bigcup_{y\in\partial\mathcal{L}_{FS}^n}O_n^{y}.
  \end{equation} 
\EEE

Here $O_n$ is the oscillatory set on  that consists of  union of stripes of width $1/\sqrt{n}$ and infinite length (following the way of construction of the set $H_n'$) that correspond to the possible positions of film atoms at the place $x_2=\frac{e_S}{\sqrt n}$ that are at distance $\frac{e_S}{\sqrt n}$ from some of substrate atoms. 
 $\nu_{H_n'}$ is a normal at the boundary.

The following lemma will help in the proof of lower-semicontinuity result. It is a simplified version of the proof of \cite[Theorem 1.1]{Yeung-et-al12} and we give it for the sake of completeness. 
We recall that  $\rho:=2/\sqrt{3}$.  
\begin{lemma} \label{lemivan1}
Let $D_n \in \mathcal{C}_n$ be such that $E_n(\mu_{D_n})$ is bounded, 	
where $\mu_{D_n}$ is the empirical measure associated with $D_n$. Let $H_n'$ be defined as above. Then we have that
$\mu_{D_n}- \rho \chi_{H_n'} \stackrel{*}{\rightharpoonup} 0$. 	
\end{lemma}
\begin{proof} 
It is easy to see that 
$$ \mu_{D_n}- \rho \chi_{H_n}  \stackrel{*}{\rightharpoonup} 0. $$
Namely for $\psi \in C_0(\mathbb{R}^2)$, \FFF where  $C_0(\Rz^2)$ denotes the set of continuous functions with compact support in $\Rz^2$,  \EEE 
 we have that 
$$
\left|\int_{\mathbb{R}^2}(\psi d \mu_{D_n}-\rho \chi_{H_n}\psi) dx\right| \leq  \frac{1}{n} \sum_{x \in D_n}\sup_{x \in D_n} \{|\psi(x)-\psi(y)|: |x-y|\leq \frac{1}{\sqrt{3n}} \}\\ 
\to  0,	
$$
as $n \to \infty$. From estimate \eqref{comparison1} and Lemma \ref{strip} we have that $\rho\chi_{H_n}-\rho \chi_{H_n'} \to 0$ strongly in $L^1$, from which we have the claim.
\end{proof}	

The following lower-semicontinuity result for the discrete energies $E_n$ is based on adapting some \PPP  ideas \EEE used in \PPP \cite{AlbDes} and \cite{FM3}. \EEE

 \begin{theorem} \label{lowerbound}

\PPP If  $\{D_n\}$ is a sequence of 
configurations such that  
$$\mu_{D_n} \stackrel{*}{\rightharpoonup} \rho\chi_{\PPP D\EEE}$$ 
weakly* with respect to the convergence of measures, where $\mu_{D_n}$ are the associated empirical measures of $D_n$  and $D\subset\Rz^2\setminus S$ is a set of finite perimeter with $|D|=1/\rho$, then
 \begin{eqnarray} \label{konacnopaolo}
  \liminf_{n \to \infty} E_n (\mu_{D_n}) &\geq& \mathcal{E}(D).
  \end{eqnarray} 
 \EEE 

 \end{theorem}

\EEE
\begin{proof}
\PPP Let  $\{D_n\}\subset\mathcal{C}_n$ be a sequence of configurations such that $\mu_{D_n} \stackrel{*}{\rightharpoonup} \rho\chi_E$ 
weakly* with respect to the convergence of measures, for \FFF a set $E\subset\Rz^2\setminus S$  of finite \EEE perimeter with $|E|=\sqrt{3}/2$. 
We focus on the case $q\neq 1$ only, since the other case is simpler.


\PPP Without loss of generality  we can assume that the limit in the  left hand side of \eqref{konacnopaolo} is reached and it is finite, and hence there exists $C>0$ such that $E_n(\mu_{D_n}) \leq C$ for every $n\in\Nz$. 
Then, by  the second assertion of Lemma \ref{strip}  there exists $C'>0$ such that $\#\partial D_n \leq C'\sqrt{n}$ for every $n\in\Nz$, from  which it follows that there exists a constant $C''>0$ such that 
  \begin{equation} \label{H1bound}
\mathcal{H}^1 (\partial H_n')<C''
	  \end{equation} 
 for every $n\in\Nz$. Therefore, by, e.g.,  Corollary \ref{corsbv2}, up to a non-relabelled subsequence, $\rho \chi_{H_n'} $ weakly converges in $SBV_{\textrm{loc}}(\mathbb{R}^2)$ to a function  $g\in SBV_{\textrm{loc}}(\mathbb{R}^2)$. Since, up to extracting an extra non-relabelled subsequence, $\mu_{D_n}- \rho \chi_{H_n'} \stackrel{*}{\rightharpoonup} 0$  as proved in  Lemma \ref{lemivan1}  and $\mu_{D_n} \stackrel{*}{\rightharpoonup} \rho\chi_E$ by hypothesis, then $g:=\rho\chi_E$ and $\rho \chi_{H_n'}\stackrel{*}{\rightharpoonup} \rho\chi_E$.

We observe that by \eqref{energy_equivalence}, \eqref{converging_energies00}, and \eqref{H2} we have that 
	  \begin{equation} \label{energy_decomposition} 
 E_n(\mu_{D_n}) = 2c_F  \mathcal{H}^1 (\partial H_n')-c_S \mathcal{H}^1\left(\partial H_n' \cap \left(\Rz\times\left\{\frac{e_S}{\sqrt n}-\frac{1}{2\PPP\sqrt{3n}\EEE}\right\}\right)\cap O_n \right)
	  \end{equation} 
\PPP where $O_n$ is the oscillation set defined in \eqref{total_oscillation}. Fix $\delta>0$ and consider in this proof the notation $y:=(y_1,y_2)\in\Rz^2$ for the coordinate of a point $y\in\Rz^2$. From \eqref{energy_decomposition} it easily follows that \EEE
\begin{eqnarray*} 
E_n (\mu_{D_n}) &=& 2c_{F} \mathcal{H}^1 (\partial H_n' \cap\{\PPP y_2\EEE>\delta\})+\UUU2c_{F}\EEE\mathcal{H}^1 (\partial H_n' \cap \{0 \leq  \PPP y_2\EEE \leq \delta\})\\ & 
&-c_S \mathcal{H}^1\left(\partial H_n' \cap \left\{ \PPP y_2\EEE=\frac{e_S}{\sqrt n}-\frac{1}{2\PPP\sqrt{3n}\EEE}\right\}\cap O_n \right)\\
&=& \int_{\partial H_n' \cap \{ \PPP y_2\EEE >\delta\} }\Gamma (\nu_{H_n'}) d \mathcal{H}^1+2c_F\mathcal{H}^1 (\partial H_n' \cap \{0 \leq  \PPP y_2\EEE \leq \delta\})\\ & 
&-c_S \mathcal{H}^1\left(\partial H_n' \cap \left\{ \PPP y_2\EEE=\frac{e_S}{\sqrt n}-\frac{1}{2\PPP\sqrt{3n}\EEE}\right\}\cap O_n \right), 
\end{eqnarray*}
where in the second equality we used the definition of $\Gamma$ (see \eqref{gamma}) to see that
$$
\Gamma(\pm\tthree)=\Gamma\left(\pm\frac{\tone+\ttwo}{\sqrt3}\right)=\Gamma\left(\pm\frac{\ttwo-2\tone}{\sqrt3}\right)=2 c_F.
$$
By Reshetnyak's lower semicontinuity \cite[Theorem 2.38]{AFP} we obtain that
$$ 
\liminf_{n \to \infty}  \int_{\partial H_n' \cap \{x_2>\delta\}} \Gamma (\nu_{H_n'}) d \mathcal{H}^1 \geq \liminf_{n \to \infty}  \int_{\partial H_n'\cap  B(R) \cap \{x_2>\delta\}} \Gamma (\nu_{H_n'}) d \mathcal{H}^1 \geq \int_{\partial^* E\cap  B(R) \cap  \{x_2>\delta\}} \Gamma (\nu_E) d\mathcal{H}^1,
$$
for every \FFF ball $B(R)$ centered at the origin and with radius $R>0$, since \EEE $\chi_{H_n'}$ converges weakly* in \PPP $BV_{\textrm{loc}}(\mathbb{R}^2)$ (and thus strongly in $L^1_{\textrm{loc}}$) \EEE to $\chi_E$, \PPP and hence, \EEE
by letting $R \to \infty$,
\begin{equation} \label{lsc1}
\liminf_{n \to \infty}  \int_{\partial H_n' \cap \{x_2>\delta\}} \Gamma (\nu_{H_n'}) d \mathcal{H}^1 \geq \int_{\partial^* E \cap  \{x_2>\delta\}} \Gamma (\nu_E) d\mathcal{H}^1. 
\end{equation} 

\PPP We claim that  for all $\delta>0$ small enough  
\begin{eqnarray}
\nonumber& &\liminf_{n \to \infty} \left[2c_F \mathcal{H}^1 (\partial H_n' \cap \{0 \leq \PPP y_2\EEE \leq  \delta\})-c_S \mathcal{H}^1\left(\partial H_n' \cap \left\{ \PPP y_2\EEE=\frac{e_S}{\sqrt n}-\frac{1}{2\PPP\sqrt{3n}\EEE}\right\}\cap O_n \right) \right]  \\
& & \label{eqivan30}  \hspace {+35ex} \geq \left(2c_F-{\OOO \frac{c_S}{q}} \right) \mathcal{H}^1(\partial^*E \cap\{\PPP y_2\EEE=0\})  \label{aerodrom1}
\end{eqnarray}
\PPP and we notice that from  \eqref{lsc1}  and \eqref{aerodrom1}   we obtain \EEE
\begin{eqnarray*}
\liminf_{n \to \infty} E_n(\mu_{D_n}) &\geq & \int_{\partial^* E'\cap \{x_2>\delta\}} \Gamma (\nu_E) d\mathcal{H}^1+\left(2c_F-{\OOO\frac{c_S}{q}}\right) \mathcal{H}^1(\partial^*E \cap\{x_2=0\}),
\end{eqnarray*}	
\PPP from which \eqref{konacnopaolo} directly follows by letting $\delta \to 0$. 
To prove the claim \eqref{aerodrom1} we fix $\delta>0$, we introduce  the Borel measures $\kappa_{1,n}, \kappa_{2,n}$, and  $\kappa_n$  defined by	
\begin{eqnarray*}
	\kappa_{1,n} (B)&:=& \mathcal{H}^1 (\partial H_n' \cap \{0 \leq  y_2 \leq  \delta\}\cap B),\\ \kappa_{2,n}(B) & :=&  \mathcal{H}^1\left(\partial H_n' \cap \left\{ y_2=\frac{e_S}{\sqrt n}-\frac{1}{2\PPP\sqrt{3n}\EEE}\right\}\cap O_n \cap B\right), \\
	\kappa_n (B)&\PPP:=\EEE& 2c_F\kappa_{1,n}(B)-c_S{\kappa_{2,n}}(B), 
\end{eqnarray*}
for every $B \in \mathcal{B} (\mathbb{R}^2)$, where \FFF $\mathcal{B} (A) $ for a set $A$  denotes the Borel $\sigma$-algebra on $A$, and we consider the sets 
$$Q_{M}:=[-M,M] \times [0,\delta], \quad  \mathring{Q}_M:=(-M,M) \times [0,\delta],\quad\text{and}\quad Q_M^c:=(\Rz \times [0,\delta])\backslash  Q_M. $$
We divide the proof in three steps:\EEE

{\bf Step 1.}  
In this step we prove that for every $M>0$ we have that
\begin{equation} \label{eqivan23} 
\liminf_{n \to \infty} \kappa_n(Q_M) \geq (2c_F-\frac{c_S}{q})\mathcal{H}^1(\partial^*E\cap \{x_2=0\}\cap Q_M). 
\end{equation}
 By \eqref{H1bound} we conclude that, up to extracting a non-relabelled subsequence, for every $M>0$, there exist  Borel measures
$\kappa_1^M$ and $\kappa_2^M$ such that $\kappa_{1,n}|_{Q_M} \stackrel{*}{\rightharpoonup} \kappa_1^M$ and $\kappa_{2,n}|_{Q_M} \stackrel{*}{\rightharpoonup} \kappa_2^M$. Consequently $\kappa_n|_{Q_M} \stackrel{*}{\rightharpoonup} 2c_F \kappa_1^M-c_S \kappa_2^M$. 
By using Lemma \ref{lmabscon} we conclude that $\kappa_2^M$ is absolutely continuous with respect to  the  Borel measure 
\begin{equation}\label{abs_cont}
\mu^M(\cdot)=\mathcal{H}^1 (\{ y_2 =0\}\cap Q_M \cap \cdot)
\end{equation}
 and we denote its density with respect to $\mu^M$  by $\zeta_2^M$. The measure $\kappa_1^M$  might not be absolutely continuous with respect to \eqref{abs_cont}. We denote the density of its  absolutely continuous part with respect to $\mu^M$  by  $\zeta_1^M$. 
To  conclude  the proof of \eqref{eqivan23}
\begin{eqnarray} \label{kontvrd1}
2c_F\zeta_1^M( y_1)-c_S\zeta_2^M( y_1) &\geq & 2c_F-\frac{c_S}{q}, \textrm{ for } (y_1,0) \in \partial^{*} E \cap \mathring{Q}_M, \\ \label{kontvrd2}
 2c_F\zeta_1^M( y_1)-c_S\zeta_2^M( y_1) &\geq & 0,  \textrm{ for  $\mathcal{H}^1$ a.e. } ( y_1,0)  \in (\partial S\setminus\partial^{*} E)\cap  \mathring{Q}_M. 
\end{eqnarray}
 We begin by showing \eqref{kontvrd1}.  Take  $y'=(y'_1,0) \in \partial^* E \cap \mathring{Q}_M$ and denote by $Q_{\varepsilon}(y')$ the square  centered at $y'$ with edges of size $\varepsilon$  parallel to the coordinate axes, where $\varepsilon>0$ is small enough such that $Q_{\varepsilon}(y') \subset \mathring{Q}_M$. Let  $Q^{+}_{\varepsilon}(y'):=Q_{\varepsilon}(y)\cap \{y_2>0\}$.  By  standard properties (see, e.g., \cite[Example 3.68]{AFP}) 
 we conclude that 
$$ \lim_{\varepsilon \to 0} \frac{1}{\varepsilon^2} \int_{Q^{+}_{\varepsilon}(y') } |\chi_E(z)-1|dz=0. $$
Since $\chi_{H_n'} \to \chi_E$ \PPP as $n\to\infty$ in $L^1(Q^{+}_{\varepsilon}(y'))$ for $\varepsilon>0$ fixed \EEE we conclude that 
$$ \lim_{\varepsilon \to 0} \lim_{n \to \infty} \frac{1}{\varepsilon^2} \int_{Q^{+}_{\varepsilon}(y') } |\chi_{H_n'}(z)-1|dz=\lim_{\varepsilon \to 0}\frac{1}{\varepsilon^2} \int_{Q^{+}_{\varepsilon}(y') } |\chi_{E}(z)-1|dz=0. $$
Thus,  for every $0<\alpha<1$ there exists $0<\varepsilon_0< \frac{\delta}{2}$ such that 
\begin{equation*}
 \liminf_{n \to \infty} |H_n' \cap Q^{+}_{\varepsilon}(y)|\geq \frac{\alpha}{2}\varepsilon^2, \quad \forall \varepsilon< \varepsilon_0. 
 \end{equation*} 
Next we define the sets $Q^0_{\varepsilon}(y'):= Q_{\varepsilon}(y') \cap \{y_2=0\}$. We have that
\begin{equation} \label{estimate1}
{\OOO\liminf_{n \to \infty}  \mathcal{H}^1\left(\{(y_1,0) \in Q^0_{\varepsilon}(y'): \{y_1\} \times \mathbb{R}^{+} \cap  H_n' \neq \emptyset\} \right)\geq \varepsilon \alpha. }
\end{equation}
We look the bottom face of the set $H_n'$ and project it on $ Q^0_{\varepsilon}(y')$. From the estimate \eqref{estimate1} it follows that 
\begin{equation} \label{estimate2} 
\liminf_{n \to \infty}\kappa_n(Q_{\varepsilon} (y'))\geq{\OOO  \left(  1/q (2c_F-c_S)+(\alpha-1/q)2c_F \right) \varepsilon.}
\end{equation}
We take a sequence in $(\varepsilon)$, still denoted by $(\varepsilon)$ such that for each memeber of the sequence we have $\kappa_1(\partial Q_{\varepsilon}(x'))=\kappa_2(\partial Q_{\varepsilon}(y'))=0$.
 By the standard properties of measures  (see \cite[Section 1.6.1, Theorem 1]{Evans}) and \eqref{estimate2} we have 
\begin{eqnarray*}
2c_F \zeta_1^M(y')-c_S \zeta_2^M(y')&=& \lim_{\varepsilon \to 0} \frac{2c_F \kappa_1^M(Q_{\varepsilon}(y'))-c_S \kappa_2^M(Q_{\varepsilon}(y'))}{\varepsilon} \\ &=& \lim_{\varepsilon \to 0}  \lim_{n \to \infty}  \frac{\kappa_n(Q_{\varepsilon} (y'))}{\varepsilon} \\
&\geq & {\OOO 1/q (2c_F-c_S)+(\alpha-1/q)2c_F}. 
\end{eqnarray*} 
By letting $\alpha \to 1$ we have \eqref{kontvrd1}.

\PPP It remains to show \eqref{kontvrd2}. \EEE   Let  $y'=(y_1',0) \in \mathring{Q}_M \backslash \partial^{*} E$. Notice that by \PPP standard property of BV functions $\mathcal{H}^1$ a.e. $(y_1,0)$ that does not belong to $\partial^{*} E$, belongs to the set of density zero for $E$ (see \cite[Theorem 3.61]{AFP})\PPP, i.e., \EEE
$$ \lim_{\varepsilon \to 0} \lim_{n \to \infty} \frac{1}{\varepsilon^2} \int_{Q^{+}_{\varepsilon}(y') } \chi_{H_n'}(y)dy=\lim_{\varepsilon \to 0}\frac{1}{\varepsilon^2} \int_{Q^{+}_{\varepsilon}(y') }\chi_{E}(y)dy=0. $$
Thus\PPP, for \EEE each $\alpha>0$ there exits \PPP $\varepsilon_0>0$ \EEE such that 
\begin{equation}\label{estimate10} 
\limsup_{n \to \infty} |H_n' \cap Q^{+}_{\varepsilon}(\PPP y'\EEE)|\leq \alpha\varepsilon^2, \quad \forall \varepsilon< \varepsilon_0. 
\end{equation}  
We need to pay attention to the atoms $y'$ that are bonded with substrate atoms, whose deficiency contribution (recall \eqref{deficiency}) can be negative and as low as $2c_F-c_S$. 

The proof consists in showing that  for $n$ large enough the total ``energy deficiency" on the cube $Q_{\varepsilon} (y')$ is actually positive, since there is  ``not much of set $E$''  in the cube $Q_{\varepsilon} (y')$.
We define
$$
K_{n,\varepsilon}(y'):= \frac{\partial \mathcal{L}_{FS}}{\sqrt{n}}\cap Q_{\varepsilon} (y')  \cap D_n
$$

 Fix  $a_0\in K_{n,\varepsilon}(y')$ and denote by $a_{-1}$ and $a_{1}$ the closest points to $a_0$ in $\partial \mathcal{L}_{FS}/\sqrt{n}$ on the left and on the right of $a_0$, respectively. We consider the set
$$
 \widetilde O_n^{a_0}:=\bigcup_{i=-1,0,1}\widetilde  O_{n}^{a_0,i},
$$
where $\widetilde O_{n}^{a_0,-1}:=O_n^{a_{-1},+}$, $\widetilde O_{n}^{a_0,1}:=O_n^{a_{1},-}$, and $\widetilde O_{n}^{a_0,0}:=O_n^{a_{0}}$, and we denote its projection onto $\partial S$ by $P^{a_0}_n$. Notice  that $\mathcal{H}^1(P^{a_0}_n)=2/\sqrt{n}$. We claim \FFF that \EEE
$$ \limsup_{n \to \infty} \mathcal{H}^1 \left(\bigcup_{a_0 \in \tilde K_{n, \varepsilon}(y')} P^{a_0}_{n}\right) \leq 16\alpha \varepsilon $$
where 
$$
\widetilde K_{n,\varepsilon}(y'):= \left\{a_0\in K_{n,\varepsilon}(y')\,:\, \text{$\FFF\exists \EEE i\in\{-1,0,1\}$ such that}\, |\widetilde O_{n}^{a_0,i}\cap H_n'\cap Q_{\varepsilon} (y')|>\frac{\varepsilon}{8\sqrt n}
\right\}.
$$

\FFF Indeed, as a \EEE consequence of \eqref{estimate10} we have
\begin{equation} \label{eqivan11} 
\#\FFF \widetilde K_{n,\varepsilon}(y') \EEE \leq 8\alpha \varepsilon \sqrt{n}+1 
\end{equation}
\FFF and hence, by  \EEE \eqref{eqivan11} 
we have
\begin{equation} \label{eqivan24} 
\sum_{a_0 \in \widetilde{K}_{n,\varepsilon}(y')}\kappa_{n}^M (\widetilde O_n^{a_0}) \geq -|2c_F-c_S|(8\alpha \varepsilon \sqrt{n}+1) \frac{1}{\sqrt{n}}. 
	\end{equation}
\FFF We now fix $ a_0 \in K_{n, \varepsilon}(y') \backslash \widetilde K_{n, \varepsilon}(y') $ such that $a_0$ is  neither the first left nor the \EEE last right atom in $K_{n,\varepsilon}(y')$ \FFF and show by \EEE a simple analysis of the  atoms  $a_i$, $i=-1,0,1$ that 
 \begin{equation} \label{eqivan1} 
 \mathcal{H}^1(\partial H_n'\cap Q_{\varepsilon} (y') \cap \widetilde O_n^{a_0}) \geq \FFF \frac{3}{\sqrt{n}}\EEE
 \end{equation}
\FFF which immediately implies \EEE that for $ a_0 \in K_{n, \varepsilon}(y') \backslash \widetilde K_{n, \varepsilon}(y') $ the energy contribution of the strip $ Q_{\varepsilon} (y') \cap \widetilde O_n^{a_0}$ for every $\varepsilon>0$ is positive and \FFF so, \EEE
  \begin{equation} \label{eqivan25} 
  \sum_{a_0 \in K_{n,\varepsilon}(y')\backslash\widetilde{K}_{n,\varepsilon}(y')}\kappa_{n}^M (\widetilde O_n^{a_0}) \geq (6c_F-c_S) \frac{\# (K_{n,\varepsilon}(y')\backslash\tilde{K}_{n,\varepsilon}(y'))}{\sqrt{n}}\geq 0. 
  \end{equation}
To prove \eqref{eqivan1} we analyse \FFF the \EEE three possible cases: 
\begin{enumerate}
	\item both  of the strips  $\widetilde O_{n}^{a_0,-1}$ and $\widetilde O_{n}^{a_0,+1}$  have empty intersection with $H_n'$;
	\item  one of the strips  $\widetilde O_{n}^{\FFF a_0,-1\EEE}$ and $\widetilde O_{n}^{a_0+1}$   has empty intersection with $H_n'$;
	\item  none of the strips  $\widetilde O_{n}^{a_0,-1}$ and $\widetilde O_{n}^{a_0,+1}$   has empty intersection with $H_n'$;
\end{enumerate}
In \FFF the \EEE first case we have that  $a_0$  does not have \FFF neighbors \EEE and \FFF hence, \EEE there is a part of $\partial H_n'$ of length $3/\sqrt{n}$ (perimeter of the equilateral triangle with side of size $1/\sqrt{n}$) that surrounds $a_0$, i.e., belongs to $v(a_0)\cap \partial H_n'$.  This proves \eqref{eqivan1} in the case of (1). \FFF For the second case we suppose without loss of generality that the \EEE  interior of the strip of $\widetilde O_{n}^{a_0,-1}$ has empty intersection with $H_n'$. We take the atom $x_1^r$  that belongs to $D_n\cap \overline{\widetilde O_{n}^{a_0,+1}}$ that is the \FFF lowest \EEE  and the atom $x_2^r \in D_n\cap \overline{\widetilde O_{n}^{a_0,+1}}$ that does not have at least one of \FFF the two of \EEE his upper \FFF neighbors\EEE. Both of these atoms always exist (the second one by the fact that $a_0 \in K_{n, \varepsilon}(y') \backslash \tilde K_{n, \varepsilon}(y')$). 
It is easy to see that \eqref{eqivan1} is satisfied also in this 
case since we have contribution of $2/\sqrt{n}$ from $v(a_0) \cap \partial H_n'$, where $v$ is defined in \eqref{scaled_voronoi},  and at least $1/(2\sqrt{n})$ from 
$$\FFF\left(v(x_1^r) \cup v\left(x_1^r-\frac{1}{\sqrt{n}\ttwo}\right)\cup v\left(x_1^r+\frac{1}{\sqrt{n}(\tone-\ttwo)}\right)\right)\EEE\cap \widetilde O_{n}^{a_0,+1}  \cap \partial H_n',$$ and at least  $1/(2\sqrt{n})$ 
from $v(x_2^r)\cap \widetilde O_{n}^{a_0,+1} \cap \partial H_n'$; in the case when $x_1^r=x_2^r$ we have the contribution of at least $1/\sqrt{n}$ from  
 \FFF$$\left(v(x_1^r) \cup v\left(x_1^r-\frac{1}{\sqrt{n}\ttwo}\right)\cup v\left(x_1^r+\frac{1}{\sqrt{n}(\tone-\ttwo)}\right)\right)\EEE\cap \widetilde O_{n}^{a_0,-1}  \cap \partial H_n'.$$ In \FFF a \EEE similar way in the \FFF third  case \EEE we find atoms $x_1^l,x_2^l \in D_n \cap \overline{\widetilde O_{n}^{a_0,+1}}$ and $x_1^r,x_2^r \in D_n \cap \overline{\widetilde O_{n}^{a_0,+1}}$ \FFF for which \EEE  there exist contribution of $1/\sqrt{n}$ coming from $v(a_0) \cap \partial H_n'$,  $1/\sqrt{n}$ coming from
\begin{align*}\FFF
\Bigg[\left(v(x_1^r) \cup v\left(x_1^r-\frac{1}{\sqrt{n}\ttwo}\right)\cup v\left(x_1^r+\frac{1}{\sqrt{n}(\tone-\ttwo)}\right)\right)\cap &\FFF \widetilde O_{n}^{a_0,+1}\cap \partial H_n'\Bigg]\\
&\FFF\cup (v(x_2^r)\cap \widetilde O_{n}^{a_0,+1} \cap \partial H_n'),
\end{align*}\EEE
 and  $1/\sqrt{n}$ coming from 
 \begin{align*}\FFF
 \Bigg[
 \left(v(x_1^l) \cup v\left(x_1^l-\frac{1}{\sqrt{n}\ttwo}\right)\cup v\left(x_1^l+\frac{1}{\sqrt{n}(\tone-\ttwo)}\right)\right)&\cap \widetilde O_{n}^{a_0,-1} \cap \partial H_n'
 \Bigg]\\
&\FFF \cup (v(x_2^l)\cap \widetilde O_{n}^{a_0,-1} \cap \partial H_n'). 
\end{align*}\EEE
The rest of the energy deficiency that is inside the strip is positive.  From \eqref{eqivan24} and \eqref{eqivan25} we conclude that 
\begin{eqnarray*}
	2c_F \zeta^M_1(y')-c_S \zeta^M_2(y')&=& \lim_{\varepsilon \to 0} \frac{2c_F \kappa_1^M(Q_{\varepsilon}(y'))-c_S \kappa_2^M(Q_{\varepsilon}(y'))}{\varepsilon} \\ &=& \lim_{\varepsilon \to 0}  \lim_{n \to \infty}  \frac{\kappa_n(Q_{\varepsilon} (y'))}{\varepsilon} \\
	&\geq & -8|2c_F-c_S| \alpha. 
\end{eqnarray*} 
By letting $\alpha \to 0$, \FFF \eqref{kontvrd2} follows. \EEE

\PPP {\bf Step 2.}  In this step we deduce \eqref{aerodrom1} from the inequalities \eqref{kontvrd1} and \eqref{kontvrd2} proved in Step 1. It suffices \EEE to show that \PPP for every $ \varepsilon>0$ there exist $M_0>0$ and $n_0 \in \mathbb{N}$ such that  
\begin{equation} \label{aerodrom3} 
\kappa_n (Q_M^c) \geq -\varepsilon 
\end{equation} 		
for every $M \geq M_0$.   
To establish \PPP \eqref{aerodrom3} fix \EEE  $\varepsilon>0$ \PPP and  choose \EEE $M_0>0$ and 
$$n_0 > \frac{16}{\varepsilon^2c_S^2}+1$$
\PPP large enough  \FFF so \EEE that the following three assertions hold: \EEE
\begin{enumerate} 
\item $|E\cap Q_{M_0}^c| \leq \frac{1}{48c_S}\varepsilon \delta$,
\item $\left|(E\cap Q_{M_0}) \Delta (H_n\cap Q_{M_0})\right| \leq \frac{1}{48c_S} \varepsilon \delta$, $\forall n \geq n_0$, 	
\item $|H_n \Delta H_n'| \leq \frac{1}{48c_S}\varepsilon\delta $, $\forall n \geq n_0$.
\end{enumerate} 	
\PPP Notice that such $M_0$ and $n_0$ exist since (1) is trivial for large $M_0$, (2) follows from the \FFF $BV_{\textrm{loc}}$-convergence \EEE of \EEE $\rho\chi_{H_n}$  to $\rho \chi_E$, and  (3) \PPP is a \EEE consequence of \eqref{comparison1}. \PPP 

By \EEE (2) and (3) and the fact that 
$$|H_n|=|E|=\frac{\sqrt 3}{2},$$
 we have that 
\begin{equation} \label{aerodrom2} 
|H_n' \cap Q_M^c| \leq \frac{1}{16c_S} \varepsilon \delta, \quad \forall M \geq M_0, \forall n \geq n_0.
\end{equation}

 We define 
$$
K_{n,M}:= \frac{\partial \mathcal{L}_{FS}}{\sqrt{n}}\cap Q_M^c \cap D_n. 
$$
Following the same idea of the previous step the proof consists in using the fact that  \FFF ``there is not much of the set $E$ outside $Q_M$'' \EEE  and hence, the  energy deficiency outside $Q_M$ is small, for $n$ large enough.

From \eqref{aerodrom2} it follows that the set $\widetilde K_{n,M}$ defined by 
\begin{align*}
\widetilde K_{n,M}:= \left\{a\in K_{n,M}\,:\, \text{$\exists\alpha\in\{-1,+1\}$ such that}\, |\widetilde O_{n}^{a,\alpha}\cap H_n'\cap (\mathbb{R} \times [0,\delta])|>\frac{\delta}{8\sqrt n}
\right\}\\ 
\end{align*}
is such that 
$$ \#\widetilde K_{n,M} \leq \frac{1}{2c_S} \sqrt{n} \varepsilon+2, 
$$
and hence
\begin{equation}\label{last_lower}
 \mathcal{H}^1\left( \bigcup_{a \in \widetilde K_{n,M}} \left(a-\frac{1}{2 \sqrt n}, a+\frac{1}{2 \sqrt n}\right) \right) \leq  \frac{\varepsilon}{c_S}
 \end{equation}
for every $n \geq n_0$.   
Since following the same argumentation of the previous step the energy deficiency associated to points in $K_{n,M} \backslash \widetilde K_{n,M}$ is shown to be positive,  from \eqref{last_lower} we easily  conclude \eqref{aerodrom3} 
for every $M \geq M_0$ and $n \geq n_0$.

{\bf Step 3}. 	
Claim \eqref{eqivan30} is an easy consequence of Step 1 and Step 2.  \FFF More precisely, \EEE  from Step 1 and Step 2  we have that  for every $\varepsilon>0$  there exists $M_0>0$ such that
\begin{eqnarray*} 
	\liminf_{n\to \infty} \kappa_n (\Rz^2)&\geq& \liminf_{n \to \infty} \kappa_n(Q_M) +\liminf_{n \to \infty} \kappa_n(Q_M^c)\\ 
	&\geq &  \left(2c_F-\frac{c_S}{q}\right)\mathcal{H}^1(\partial^*E\cap \{x_2=0\}\cap Q_M)-\varepsilon   
\end{eqnarray*}	
for any $M \geq M_0$, where we used \eqref{eqivan23} and \eqref{aerodrom3}. By letting $M \to \infty$ and using arbitrariness of $\varepsilon>0$ \FFF we obtain \eqref{eqivan30}.  
	
\EEE

\end{proof}

\RRR
\section{Upper bound}\label{sec:upper_bound} 
The proof of \PPP the \EEE upper bound  follows \PPP from the arguments of \EEE \cite{Yeung-et-al12} \PPP by playing extra care to the contact with the substrate. \EEE

\begin{theorem} \label{thm:upperbound} 
\PPP For \EEE every set $D \subset \Rz^2\setminus S$  of finite perimeter such that $|D|=\sqrt{3}/2$ there exists a sequence of configurations \PPP $D_n \in\mathcal{C}_n$ \EEE such that the \PPP corresponding \EEE associated empirical measures $\mu_{D_n}$    weak* converge to $\rho \chi_D$  and $I_n (\mu_{D_n}) \to \mathcal{E}(D)$.\EEE
\end{theorem}
\begin{proof}
\PPP
The proof is divided in 5  steps.
\EEE

\textbf{Step 1 \PPP (Approximation by bounded smooth sets).}  \EEE
	In this step we claim that: If $E  \subset \Rz^2\setminus S$  is a set of finite perimeter with $|E|=1/\rho$, then  there \FFF exists a \EEE sequence of sets $(E_j)_{j \in \mathbf{N}}$ \FFF with \EEE $E_j \subset \Rz^2\setminus S$ \FFF for \EEE $j \in \mathbf{N}$  such that \PPP the following assertions hold: 
\begin{enumerate} 
\item[(i)] 	$|E_j|=  1/\rho $; 
\item[(ii)] $E_j$ \FFF are \EEE bounded;
\item[(iii)]   there exist sets  $E_j' \subset \mathbb{R}^2$ of class $\mathcal{C}^{\infty}$ such that $E_j=E_j' \cap(\Rz^2\setminus \bar{S})$; 
\item[(iv)]  $|E_j \Delta E| \to 0$ as $j \to \infty$;
\item[(v)] $|D \chi_{E_j}| (\Rz^2\setminus \bar{S}) \to | D \chi_E | (\Rz^2\setminus \overline{S})$ as $j \to \infty$;
\item[(vi)] \begin{equation}\label{for_konacnopaolo1}
\int_{\partial^* E_j \cap \PPP(\Rz^2\setminus \bar{S})\EEE} \Gamma (\nu_{E_j} ) d \mathcal{H}^1 \to  \int_{\partial^* E \cap (\Rz^2\setminus \overline{S})} \Gamma (\nu_{E} ) d \mathcal{H}^1 \textrm{ as } j \to \infty;
\end{equation} 
\item[(vii)] \begin{equation}\label{for_konacnopaolo2}
\mathcal{H}^1 (\partial^{*} E_j \cap \partial S) \to \mathcal{H}^1 (\partial^{*} E \cap \partial S) \textrm{ as } j \to \infty.
\end{equation} 
\end{enumerate} 	
We \FFF now  \EEE construct the sequence of sets $(E_j)_{j \in \mathbf{N}}$ that satisfy (ii)-(vii) \FFF and observe that then \EEE (i) is  easily obtained by scaling.
 Let $E' \subset \mathbb{R}^2$ be the set determined from $E$ by reflection over $\partial S$ and note that 
\begin{equation} \label{reflection}
\mathcal{H}^1(\partial^*E'\cap \partial S)=0.
\end{equation} 
By  \cite[Theorem 13.8 and Remark 13.9]{Maggi} we find smooth bounded  open sets $ E_j'\subset \mathbb{R}^2$ that satisfy $| E_j' \Delta E'|\to 0$ and 
\begin{equation} \label{convergence_perimeters}
 |D\chi_{ E_j'}| (\mathbb{R}^2) \to |D\chi_{E'}| (\mathbb{R}^2).
\end{equation} 
We define $ E_j:= E_j' \cap (\Rz^2\setminus \bar{S})$  and we  claim that the sets $E_j$   satisfy (ii)-(vii). 

We begin by noticing that (ii)-(iv) are trivial. 
\PPP To prove assertion  (v) we begin to observe that
\begin{eqnarray}
|D \chi_{E'}|(\Rz^2\setminus \overline{S}) &\leq& \liminf_{j \to \infty} |D \chi_{E_j'}|(\Rz^2\setminus \overline{S}), \\ 
|D \chi_{E'}|(S) &\leq& \liminf_{j \to \infty} |D \chi_{E_j'}|(S),\\
|D \chi_{E'}|(\partial S)=0 &\leq& \liminf_{j \to \infty} |D \chi_{E_j'}|(\partial S),
 \end{eqnarray}
where we used  \eqref{convergence_perimeters} and \eqref{reflection}, and hence,
\begin{eqnarray*}
&|D \chi_{E'}|(\Rz^2)=|D \chi_{E'}|(\Rz^2\setminus\overline{S})+|D \chi_{E'}|(S)+  |D \chi_{E'}|(\partial S)\\
&\leq \liminf_{j \to \infty} |D \chi_{E_j'}|(\Rz^2\setminus \overline{S})+ \liminf_{j \to \infty} |D \chi_{E_j'}|(S)+ \liminf_{j \to \infty} |D \chi_{E_j'}|(\partial S)\\
&\leq \liminf_{j \to \infty} |D \chi_{E_j'}|(\Rz^2)=|D \chi_{E'}|(\Rz^2).
 \end{eqnarray*}
 Since this can be done on an arbitrary subsequence we have
 \begin{equation} \label{eqivan44} 
 |D \chi_{E}|(\Rz^2\setminus \overline{S})= |D \chi_{E'}|(\Rz^2\setminus \overline{S})= \lim_{j \to \infty} |D \chi_{E_j'}|(\Rz^2\setminus \overline{S})= \lim_{j \to \infty} |D \chi_{E_j}|(\Rz^2\setminus \overline{S}),
 \end{equation} 
 since by the definition of $E'$ and  $E_j$ we have that $D \chi_{E'}(A) = D \chi_{E}(A)$  and $D\chi_{E_j'}(A) = D \chi_{E_j}(A)$ for every $A\in\mathcal{B}(\Rz^2\setminus \overline{S})$. \EEE 

Assertion \EEE (vi) is a direct consequence of Reshetnyak continuity theorem \cite[Theorem 2.39]{AFP}.

To prove  assertion  (vii) we  will first \FFF claim \EEE that for almost every $M>0$  
\begin{equation}\label{for_konacnopaolo2}
\mathcal{H}^1 (\partial E_j \cap  (-M,M) \times\{0\}) \to \mathcal{H}^1 (\partial^{*} E\cap    (-M,M)\times\{0\}).
\end{equation} 
To prove  \FFF claim \EEE \eqref{for_konacnopaolo2} we observe that for almost every $M \geq 0$ we have 
$$
|D \chi_{E'}| \left((\{-M\}\times \mathbb{R} )\cup (\{M\} \times \mathbb{R})\right)=0. 
$$ 
In fact the set of $M\geq0$ where this condition is not satisfied is at most countable. As in the proof of (v) we conclude that for all such $M \geq 0$ we have 
\begin{eqnarray} 
\label{eqivan51} |D\chi_{ E_j}|\left((-\infty, -M)  \times \mathbb{R}^+\right) &\to& |D\chi_{E}|\left((-\infty, -M)  \times \mathbb{R}^+\right), \\ \label{eqivan50} |D\chi_{E_j}|\left( (M, +\infty )  \times \mathbb{R}^+\right) &\to& |D\chi_{E}|\left(( M, +\infty)  \times \mathbb{R}^+\right), \\ \label{eqivan47}   |D\chi_{E_j}|\left(( -M, M )  \times \mathbb{R}^+\right) &\to& |D\chi_{E}|\left( (-M, M )  \times \mathbb{R}^+\right).
\end{eqnarray} 
\FFF We then notice that \EEE \eqref{for_konacnopaolo2} is a consequence of continuity of traces and \eqref{eqivan47}. 

We now make the second claim that 
 \begin{eqnarray} \label{eqivan48}
 \mathcal{H}^1\left(\partial E_j \cap (-\infty, -M)  \times \{0\} \right)&\to & 0, \quad \textrm{ as } M \to \infty,\\ \label{eqivan49} 
\mathcal{H}^1\left(\partial  E_j \cap ( M, +\infty)  \times \{0\} \right)&\to & 0,\quad \textrm{ as } M \to \infty, 
\end{eqnarray} 
which together with \eqref{for_konacnopaolo2} yields (vii).  We \EEE prove only \eqref{eqivan48}, since \eqref{eqivan49} goes in an analogous way. It is enough to show that
\begin{equation} \label{eqivan52} 
 \mathcal{H}^1 \left(\partial E_j \cap  (-\infty,-M) \times\{0\}\right) \leq  
|D \chi_{ E_j}| (\langle-\infty,-M\rangle \times \mathbb{R}^+), 
\end{equation} 
since as a consequence of \eqref{eqivan51} we have that \FFF the \EEE right hand side of \eqref{eqivan52} goes to zero as $j \to \infty$. \FFF 
Estimate \EEE \eqref{eqivan52}  can be seen by taking $\varphi \in C_c(\mathbb{R}^2)$, $\varphi=\tthree$ on some open set $F$ such that $E_j \cap (-\infty, -M)  \times \mathbb{R}^+) \subset \subset F$,  in the identity 
\begin{eqnarray*} 
	\int_{\partial \left(E_j \cap (-\infty, -M)  \times \mathbb{R}^+\right)} \langle \nu_{E_j \cap (-\infty, -M)  \times \mathbb{R}^+},\varphi \rangle dx&=& \int_{F} \langle D\chi_{E_j \cap \left((-\infty, -M)  \times \mathbb{R}^+\right)},\varphi\rangle dx \\ &=& \int_F \chi_{E_j \cap \left((-\infty, -M)  \times \mathbb{R}^+\right)} \textrm{div} \varphi dx =0.
\end{eqnarray*} 
From this it follows that 
\FFF \begin{align} 
 \int_{\partial \left( E_j \cap (-\infty, -M)  \times \mathbb{R}^+\right) \cap\{x_2>0\}}  \langle & \nu_{E_j \cap (-\infty, -M)  \times \mathbb{R}^+}, \tthree  \rangle d \mathcal{H}^1\notag\\
  &=-\mathcal{H}^1 \left(\partial \left( E_j \cap (-\infty, -M)  \times \mathbb{R}^+\right) \cap \{x_2=0\}\right), \label{eqivan100} 
 \end{align} \EEE
where we used the fact that 
$$ \nu_{E_j \cap (-\infty, -M)  \times \mathbb{R}^+}=-\tthree,\quad \mathcal{H}^1 \textrm{ almost everywhere on } x_2=0. $$
\FFF Therefore,  \FFF by \eqref{eqivan100} and since \EEE
$$ \nu_{E_j \cap (-\infty, -M)  \times \mathbb{R}^+}=\tone,\quad \mathcal{H}^1 \textrm{ almost everywhere on } x_1=-M. $$
\EEE
we obtain 
\FFF \begin{align} 
& \mathcal{H}^1 \left( \partial \left( E_j \cap (-\infty, -M)  \times \mathbb{R}^+\right) \cap \{x_2=0\} \right) \notag\\
& \leq  \mathcal{H}^1 \left( \partial \left( E_j \cap (-\infty, -M) \times \mathbb{R}^+\right) \cap(-\infty,-M)\times \mathbb{R}^+  \right). \label{eqivan60} \EEE
 \end{align} 
Notice that 
\begin{equation} \label{eqivan61} 
 \partial \left( E_j \cap (\langle-\infty, -M\rangle  \times \mathbb{R}^+)\right) \cap(-\infty,-M)\times \mathbb{R}^+ =\partial E_j \cap (-\infty,-M)\times \mathbb{R}^+, 
 \end{equation} 
which together with \eqref{eqivan60}  implies \eqref{eqivan48}, since
$$ \mathcal{H}^1 \left( \partial \left( E_j \cap (-\infty, -M) \times \mathbb{R}^+\right) \cap (-\infty, -M)\times \mathbb{R}^+ \right)= |D \chi_{E_j}| (\langle-\infty, -M\rangle  \times \mathbb{R}^+), $$
see \cite[Chapter 3.3]{AFP}. This \FFF concludes \EEE  the proof of (vii).

\vspace{0.5cm}

\textbf{Step 2 \PPP(Approximation  by polygons)\EEE.}  
By Step 1  we can assume that $E \subset \subset B(R)$ is smooth and bounded. 
Furthermore, for such $E$ we can construct a sequence  of approximating polygons $P_{j}$ by choosing the vertices of each $P_{j}$ on the boundary of $E$ in such a way that 
$|P_{j} \Delta E| \to 0$,
$$ \int_{\partial P_{j} \cap \PPP (\Rz^2\setminus \overline{S}) \EEE} \Gamma (\nu_{P_{j}} ) d \mathcal{H}^1 \to  \int_{\partial E \cap (\Rz^2\setminus \overline{S})} \Gamma (\nu_{E} ) d \mathcal{H}^1,$$
and 
$$
 \mathcal{H}^1 (\partial P_{j} \cap \PPP \partial S\EEE) \to \mathcal{H}^1 (\partial E \cap \PPP \partial S\EEE),
	$$
so that
\begin{equation}\label{approximation2}
 \PPP \mathcal{E}\EEE (P_{j}) \to  \PPP \mathcal{E}\EEE  (E).
\end{equation}  
\vspace{0.5cm}
\textbf{Step 3 \PPP (Approximation by polygons with vertices on  the lattice)\EEE.}  \PPP In view \EEE of previous steps \PPP and the \EEE metrizability of the unit ball of measures (where the norm is given by total variation)  induced  by the weak* convergence\PPP, \EEE by employing \PPP a \EEE standard diagonal argument and \PPP \eqref{approximation2}\EEE, we can assume, without loss of generality, that $E$ has polygonal boundary. We now approximate \PPP such polygonal set $E$, whose number of vertices we denote by $m\in\mathbb{N}$ with \EEE a sequence of polygons $E_n$ characterized by $m$ vertices  belonging to 
$$\mathcal{L}_F^n:=\frac{1}{\sqrt{n}} \mathcal{L}_F.$$

More precisely, let $E_n$ be the polygon with vertices the set of $m$ points in $\frac{1}{\sqrt{n}} \mathcal{L}_F$ closest in the Euclidean norm to the $m$ vertices of $E$. Notice that the angles at the vertices of $E_n$ approximate the angles at the vertices of $E$, $|E_n \Delta E| \to 0$,  \RRR
$$ \int_{\partial E_{n} \cap  \PPP (\Rz^2\setminus \overline{S_n}) \EEE} \Gamma (\nu_{E_{n}} ) d \mathcal{H}^1 \to  \int_{\partial E \cap \PPP (\Rz^2\setminus \overline{S}) \EEE} \Gamma (\nu_{E} ) d \mathcal{H}^1
$$
and  \PPP
$$
\mathcal{H}^1 (\partial E_{n} \cap\PPP \partial S_n\EEE ) \to \mathcal{H}^1 (\partial E \cap  \PPP \partial S \EEE).   $$
where  $S_n$ is defined in \eqref{substrate_n}. 
Therefore, 
\begin{equation}\label{approximation3}
\PPP \mathcal{E}_{n}\EEE (E_n) \to \PPP \mathcal{E}\EEE (E),
\end{equation}  
\PPP where $\mathcal{E}_{n}$ is defined in \eqref{continuum_n}. \EEE
Furthermore, there exist    \PPP $\alpha_n\searrow0$ and  $\beta_n \searrow 0$  \EEE  such that 
\begin{equation}\label{alphabeta}
\PPP \left||E_n|-\frac{\sqrt{3}}{2}\right|=\alpha_n\qquad\textrm{and}\qquad |\mathcal{H}^1(\partial E_n)-\mathcal{H}^1(\partial E)|=\beta_n.  \end{equation}  \RRR
\vspace{0.5cm}

\textbf{\PPP Step 4 \PPP (Discrete recovery sequence)\EEE.}  \PPP  \GGG Let us now consider the sequence of crystalline configurations  $\widetilde D_n :=\mathcal{L}_F^n \cap  \overline{E}_n$, and notice that $\mu_{\widetilde D_n}$ weakly* converges to $\rho \chi_E$.  \PPP Furthermore, from the definition of   scaled Voronoi cells $v(x)$ of $x$ (see \eqref{scaled_voronoi}) it follows that\EEE
$$
 \#\widetilde D_n-n=\frac{2n}{\sqrt{3}}\sum_{x\in \widetilde D_n} |v(x)|\,-\,n=\frac{2n}{\sqrt{3}}\sum_{x\in  \PPP\widetilde D_n\setminus \partial \widetilde D_n\EEE} |\PPP v(x)\EEE|\,-\,n\,+\,\frac{2n}{\sqrt{3}}\sum_{x\in \PPP\partial \widetilde D_n\EEE} |\PPP v(x)\EEE|,
$$
and hence, since for every $x\in \widetilde D_n$ we have $  |v(x)| = \sqrt{3}/(2n)$, 
 \begin{align} \label{paolo1} 
 |\#\widetilde D_n-n|&\leq \PPP \frac{2}{\sqrt{3}}\EEE\alpha_nn + \PPP\#\partial \widetilde D_n\EEE  \notag\\
&\leq\frac{2}{\sqrt{3}}\alpha_nn +  C(\mathcal{H}^1(\partial E) +\beta_n)\sqrt{n}
  \end{align} 
\PPP  for some constant $C>0$, where in the last inequality we used \eqref{alphabeta}.\EEE

 We now claim that  
   \begin{equation} \label{paolo3}   
\frac{\PPP\widetilde V_n \EEE (\sqrt{n}\widetilde D_n)+6c_F \#\widetilde D_n}{\sqrt n}=\PPP\mathcal{E}_n\EEE (E_n),    
 \end{equation}  
 \PPP where $\widetilde V_n$ is the generalization of $V_n$ (see \eqref{V}) to configurations with a number of atoms different than $n$, i.e., 
 $$
 \widetilde V_n (D_k) := \sum_{i\neq j} v_{FF}(|d_i-d_j|)\,+\,  \sum_{i=1}^k v^1(d_i), 
 $$
 for every configuration $D_k:=\{d_1,\dots,d_k\}  \in \mathcal{C}_k$. \EEE
 The claim easily follows from the observation that each side $S_{k,n}$ of $E_n$, $k=1,\dots,m$, for $n$ large enough,  intersects  
  $\sqrt n \,\Gamma (\nu_{S_{k,n}}) \mathcal{H}^1(S_{k,n})+O(1)$ segments such that $|z_1-z_2|=1/\sqrt n$ and $(z_1,z_2) \in  \widetilde D_n \times (\PPP\mathcal{L}_F^n\EEE \backslash \widetilde D_n)$.
 To \FFF see that \EEE we begin by considering a segment $\PPP L\EEE=(x,y)$ with endpoints $x,y \in  \mathcal{L}_F$. We denote the unit tangential and normal vector to $\PPP L\EEE$ by \PPP${\boldsymbol t_L}$ \EEE and  $\nu_{\PPP L\EEE}$, respectively. Obviously $y=x+{\boldsymbol t}$ for the vector ${\boldsymbol t}:=\PPP \mathcal{H}^1(L)\EEE {\boldsymbol t_L}=k_1 \tone+ k_2 \ttwo$ defined for some $k_1, k_2 \in \mathbb{Z}$. We restrict to the case in which $k_1,k_2 \in \mathbb{N}_0$ since the remaining case can be treated analogously.  Let $\overline{\Gamma}$ be \PPP the function \EEE such that $\overline{\Gamma}({\boldsymbol t_L})=\Gamma(\nu_L)/\PPP c_F\EEE$, i.e., 
 $$\overline{\Gamma}({\boldsymbol t_L})\PPP:=\EEE 2\left({\boldsymbol t_L^1} +\frac{{\boldsymbol t_L^2}}{\sqrt 3}\right)$$  
 for 
 $${\boldsymbol t_L}:={{\boldsymbol t_L^1} \choose {\boldsymbol t_L^2}},$$
and extend $\overline{\Gamma}$ by homogeneity. Notice that
    \begin{equation} \label{gamma12} 
      \overline{\Gamma}(k_1\tone)=2k_1 \qquad\textrm{and}\qquad \overline{\Gamma}(k_2\ttwo)=2k_2 . 
  \end{equation}  
 Let $P_L$ be the parallelogram with sides the vectors $x+k_1 \tone$ and $x+k_2\ttwo$. Furthermore, let  $P_L^+$ and $P_L^-$ the open triangles in which $L$ divides $P_L$. Notice that  inside $P_L$ we have $k_1$ lines parallel to $k_2\ttwo$, $k_2$ lines parallel to $k_1\tone$, and $k_1+k_2-1$ lines with varying length that are  parallel to the vector $\ttwo-\tone$. Since $x+{\boldsymbol t}$ intersects each of these last lines (and each line intersects $L$ one time), we have that $L$ exactly intersects 
 $$
 2(k_1+k_2)-1= \overline{\Gamma}(k_1\tone)+\overline{\Gamma}(k_2\ttwo)-1=\overline{\Gamma}({\boldsymbol t})-1=\PPP \mathcal{H}^1(L)\EEE\overline{\Gamma}({\boldsymbol t_L})-1=\PPP \frac{\mathcal{H}^1(L)}{\PPP c_F\EEE}\EEE\Gamma(\nu_L) -1 
 $$ 
 lines and hence,  $L$ intersects $ \mathcal{H}^1(L)\Gamma(\nu_L)$ segments $[z_1,z_2]$ such that $|z_1-z_2|=1$, $z_1\in  \mathcal{L}_F\cap(P_L^+\cup L)$, and  $z_2\in  \mathcal{L}_F\cap P_L^-$. 
 Therefore, if we denote the $m$ vertices of $E_n$ by $v_{k,n}$ for $k=1,\dots,m$ and let $v_{m+1,n}=v_{1,n}$, then, for $n$ large enough, each side $S_{k,n}=[v_{k,n},v_{k+1,n}]$ of $E_n$   intersects  $\sqrt n \Gamma (\nu_{S_{k,n}}) |S_{k,n} |+O(1)$ segments such that $|z_1-z_2|=\frac{1}{\sqrt n}$ and $(z_1,z_2) \in  \widetilde D_n \times (  \frac{1}{\sqrt n}\mathcal{L}_F \backslash \widetilde D_n)$, where the contribution $O(1)$ takes into account  that the endpoints of $S_{k,n} $ might have a different  numbers of neighbors in $\widetilde D_n$. However such disturbance is of the order $O(1)$ \RRR since the angles of $E_n$ at the segment are approximately the same for all $n$.

\vspace{0.5cm}

\GGG
 
\textbf{\PPP Step 5 \PPP (Final recovery sequence)\EEE.}   Finally we variate the configuration $\widetilde D_n$ to obtain configurations $D_n$ such that $\# D_n=n$. 

If $\#\widetilde D_n<n$, \PPP then \GGG we can choose any set $\widehat{D}_n\subset \mathcal{L}_F^n\cap \PPP\{2R\geq x_2 >R\}\EEE$ for a fixed $ R>0$ large enough with cardinality $n-\#\widetilde D_n$ and surface energy  of order $\sqrt{n-\#\widetilde D_n}$, and define $$D_n:=\widetilde D_n\cup\widehat{D}_n.$$ 
By  \eqref{paolo1} there exists a constant $C>0$ such that  \RRR
\begin{equation} \label{paolo2}  
\left| \frac{\PPP\widetilde V_n \EEE (\sqrt{n}\widetilde D_n)+6c_F \#\widetilde D_n}{\sqrt n}-\frac{\widetilde V_n  (\sqrt{n} D_n)+6c_F n}{\sqrt n}   \right| \GGG \leq C\frac{\sqrt{|\#\widetilde D_n-n|}}{\sqrt{n}} \RRR\to 0.
 \end{equation} 
 
\GGG Similarly, if $\#\widetilde D_n>n$ by \eqref{paolo1} for $n$ large enough we can define a configuration $D_n$ satisfying \eqref{paolo2} by taking away $\PPP \#\widetilde D_n\EEE-n$ atoms from $\widetilde D_n$,  for example we can define $D_n:=\widetilde D_n\setminus P_n$ for some parallelogram $P_n\subset E_n$. 

Since $\mu_{D_n}$ weakly* converges to $\rho \chi_E$, the assertion follows from \eqref{paolo3}  and \eqref{paolo2}. \\ \RRR 

\end{proof} 
\section{Proof of the main theorems \PPP in the dewetting regime} \label{sec:main_results}\EEE

\PPP We begin the section by stating a $\Gamma$-convergence results that is a direct consequence of Sections \ref{sec:lower_bound} and \ref{sec:upper_bound}. Recall that $\rho:=2/\sqrt{3}$.

 \begin{theorem}[$\Gamma$-convergence]  \label{Thm:Gamma_convergence} 
 Assume \eqref{dewetting_condition}.  The functional 
	\begin{equation}\label{converging_energies}
	E_n:= n^{-1/2}(I_n+6c_F n),
	\end{equation}
	where $I_n$ is defined by \eqref{radon_functional}, $\Gamma$-converges with respect to the weak* convergence of measures to the functional $I_\infty$ defined by
		\begin{equation}
	I_\infty(\mu):=\begin{cases}  \mathcal{E}(D_\mu), &  \text{if $\exists$ $D_\mu\subset\Rz^2\setminus S$  set of finite perimeter}\\
	& \hspace{15ex} \text{with  $|D_\mu|=\frac{\sqrt{3}}{2}$ such that $\mu=\frac{2}{\sqrt{3}}\chi_{D_\mu}$}
\\
	+\infty, &\text{otherwise,}\\
	\end{cases}
	\end{equation}
	for every $\mu\in\mathcal{M}(\Rz^2)$. 
 \end{theorem}
 
	 \begin{proof} 
	In view of the definition of $\Gamma$-convergence the Assertion directly follows from the lower and upper bound provided by Theorems \ref{lowerbound} and  \ref{thm:upperbound}, respectively.
\end{proof}

\EEE

\PPP
We notice that Theorem \ref{Thm:Gamma_convergence}  is not enough to conclude Assertion 3. of Theorem \ref{thm:convergence_minimizers}. In fact, the compactness provided for energy equi-bounded sequences $D_n\in\mathcal{C}_n$ by Theorem \ref{compactnesstheorem} of Section \ref{sec:compactness} holds only for almost-connected configurations $D_n$. Therefore, as detailed in the following result, we can deduce the convergence of a subsequence of minimizers only after performing (for example) the  transformation $\mathcal{T}$ given by Definition \ref{transformation}, which does not change the property of being a minimizer.

\PPP

  \begin{corollary} \label{Cor:trasformed_convergence}
  Assume \eqref{dewetting_condition}. For every sequence  of minimizers $\mu_n\in\mathcal{M}_n$ of $E_n$, there exists a \FFF \emph{(}\EEE possibly different\FFF\emph{)} \EEE sequence of minimizers $\widetilde \mu_n\in\mathcal{M}_n$ of $E_n$ that admits a subsequence  converging with respect to the weak* convergence of measures to a minimizer of  $I_\infty$ in \FFF
	\begin{align*}
\mathcal{M}_W:=\bigg\{\mu\in\mathcal{M}(\Rz^2)\ :\ \text{$\exists$ $D\subset\Rz^2\setminus S$  set of finite  perimeter,}  \text{ bounded, }  \\
	& \hspace{-35ex} \text{ with  $|D|=\frac{1}{\rho}$, and such that $\mu=\rho\chi_D$}\bigg\}. 
	\end{align*} \EEE
    \end{corollary}

\begin{proof}
Let $\mu_n\in\mathcal{M}_n$ be minimizers   of $E_n$. By \eqref{radon_space}, \eqref{radon_functional}, and \eqref{converging_energies} there exist configurations $D_n\in\mathcal{C}_n$ such that $\mu_n:=\mu_{D_n}$. Let   $\mathcal{T}(D_n)\in\mathcal{C}_n$ be the transformed configurations associated to $D_n$  by Definition \ref{transformation}. We notice that the sequence of measures 
$$
\widehat \mu_n:=\mu_{\mathcal{T}(D_n)}
$$
is also a sequence of minimizers of $E_n$, since by Definition \ref{transformation} and \eqref{energy_equivalence} we have that
$$
E_n(\widehat \mu_n)\leq E_n(\mu_n). 
$$
Therefore, by Theorem \ref{Thm:Gamma_convergence} and  Theorem \ref{compactnesstheorem} we obtain that there exist  a sequence of vectors  $a_n:=t_n\tone$ for $t_n\in\Zz$, an increasing sequence $n_k$, $k\in\Nz$, and a measure $\mu\in \mathcal{M}_W$   such that   $\widetilde \mu_{n_k}\rightharpoonup^*\mu$ in $\mathcal{M}(\Rz^2)$, where
$$
\widetilde \mu_n:=\widehat \mu_n(\cdot+a_{n}).
$$
\end{proof}

\EEE

\PPP In view of Theorem \ref{connectness} we can improve the previous result and in turns, prove the convergence of minimizers (up to a subsequence) directly without passing to an  auxiliary sequence of minimizers obtained by performing the transformation $\mathcal{T}$ given by Definition \ref{transformation}. In fact, Theorem \ref{connectness} allows to exclude the possibility that a sequence of (not almost-connected) minimizers $\mu_n\in\mathcal{M}_n$ losses mass in the limit. 


\EEE
 	 \begin{proof}[Proof of Theorem \ref{connectness}]
	 \PPP Let $\widehat{D}_n$ be such that 
$$V_n(\widehat{D}_n)=\min_{D_n\in\mathcal{C}_n}{V_n(D_n)},$$
and select for every $\widehat{D}_n$ a connected component $\widehat{D}_{n,1}\subset \widehat{D}_n$ with  largest cardinality. 
\EEE
  We assume by contradiction that 
  $$\liminf_{n \to \infty}   \mu_{\widehat D_{n}} (\widehat D_{n,1})<1,$$ 
   and we select a subsequence \UUU $n_k$  such that 
\begin{equation}\label{contradiction}
 \lim_{k \to \infty} \mu_{\widehat D_{n_k}} (\widehat D_{n_k,1})=\liminf_{n \to \infty}   \mu_{\widehat D_{n}} (\widehat D_{n,1})<1.
 	\end{equation}
\PPP	
	By Corollary \ref{Cor:trasformed_convergence} there exists  a (possibly different) sequence of minimizers $\widetilde \mu_{n_k}\in\mathcal{M}_{n_k}$ of $E_{n_k}$ that (up to passing to a non-relabelled subsequence)  converge with respect to the weak* convergence of measures to a minimizer $\mu\in\mathcal{M}_W$ of  $I_\infty$. Therefore,  there exists a bounded   set $D\subset\Rz^2\setminus S$ of finite  perimeter with $|D|=1/\rho$ such that $\mu= \rho  \chi_D$ and $\widetilde \mu_{n_k}$ converge with respect to the weak* convergence to $ \rho\chi_D$.
	\EEE
	
	We claim that 	
	\begin{equation}\label{claim_positive}
	m_0:=\mathcal{E}(D)>0,
	\end{equation}
	and observe that \eqref{claim_positive} follows from 
	\begin{equation} 
	\label{nada1}
	\int_{\partial^* D \cap \{x_2>0\}} \Gamma (\nu_{D} ) d \mathcal{H}^1  \geq  \PPP 2c_F\EEE\mathcal{H}^1(\partial^*D \cap\{x_2=0\}). 
	\end{equation} 
In order to prove \eqref{nada1} we  first show that
\begin{equation} 
	\label{normal_around} 
	\int_{\partial^* D} \nu_D d \mathcal{H}^1=0, 
		\end{equation} 
by taking $\varphi_i \in C_c^{\PPP1\EEE}(\mathbf{R}^2\PPP;\mathbf{R}^2\EEE)$, $\varphi_i=\ti$ on some  open  set $F$ such that $D\subset \subset F$, for $i=1,3$, in the identity 
$$   \int_{\partial^* D} \langle \nu_D,\varphi_i \rangle \PPP d \mathcal{H}^1\EEE= \PPP-\EEE \int_{D}\textrm{div} \varphi_i dx \PPP=-\int_{D}\textrm{div} \ti dx \EEE=0,$$
where we used   the definition of reduced boundary and the  \emph{generalized Gauss-Green formula} \cite[Theorem 3.36]{AFP} for sets of finite perimeter. \EEE
Then, from \eqref{normal_around} it follows that 
$$ \int_{\partial^* D \cap \{x_2>0\}} \nu_D d\mathcal{H}^1 =-\mathcal{H}^1 (\PPP\partial^* D\EEE \cap \{x_2=0\})\tthree$$
and hence, \PPP since $\Gamma$ is convex and homogeneous, by Jensen's inequality and the fact that $\Gamma(\tthree)=2c_F$  we conclude that \EEE
$$  \int_{\partial^* D \cap \{x_2>0\}} \Gamma (\nu_{D} ) d \mathcal{H}^1 \geq L \Gamma ( \frac{1}{L}\int_{\partial^*`D \cap \{x_2>0\}}\nu_{D} d\mathcal{H}^1 )=  \PPP 2c_F\EEE\mathcal{H}^1 (\PPP\partial^* D\EEE \cap \{x_2=0\}), $$
where $L\PPP:=\EEE\mathcal{H}^1(\partial^*D \cap \{x_2>0\})$, \PPP which is \eqref{nada1}. \EEE

\FFF
We claim that there exist configurations $\widetilde D_{n_k}\in\mathcal{C}_{n_k}$ defined by
$$
\widetilde D_{n_k}:=\widetilde D_{n_k}^1\cup \widetilde D_{n_k}^2, 
$$
where $\widetilde D_{n_k}^1$ and $\widetilde D_{n_k}^2$ are configurations such that:
\begin{itemize}
\item[(i)] $\textrm{supp } \mu_{\widetilde D_{n_k}^1}\subset B(x_1,R_1)$ and $\textrm{supp } \mu_{\widetilde D_{n_k}^1}\subset B(x_2,R_2)$ for some   $x_1,x_2\in\Rz^2$ and $R_1,R_2>0$ with $ B(x_1,R_1) \cap B(x_2,R_2)=\emptyset$, 
\item[(ii)] the energy is preserved, i.e.,  $$ 
 V_{n_k}(\widetilde D_{n_k})=V_{n_k}(\widehat D_{n_k}),
$$
 \item[(iii)]  the following inequalities hold: 
 $$
\limsup_{k \to \infty } \mu_{\widetilde D_{n_k}}(\widetilde D_{n_k}^1)>0 \quad\text{and}\quad \limsup_{k \to \infty } \mu_{\widetilde D_{n_k}}(\widetilde D_{n_k}^2)>0. 
$$
\end{itemize}

Under the further assumption that
\begin{equation} \label{eqivan101} 
 \lim_{k \to \infty} \mu_{\widehat D_{n_k}} (\widehat D_{n_k,1})>0, 
\end{equation} 
we can explicitly define  $\widetilde D_{n_k}^1:=\FFF\mathcal{T}(\widehat{D}_{n_k} \backslash \widehat{D}_{n_k,1})\EEE$  and  $\widetilde D_{n_k}^2:=\widehat D_{n_k,1}+tq\tone$ for some large $t\in\Zz$ (see \eqref{eSratio} for the definition of $q$). In fact,  the configurations  $\widetilde D_{n_k}^1$ and $\widetilde D_{n_k}^2$ are bounded because by Definition \ref{transformation} they are almost connected and hence, property (i) is satisfied provided that  $t\in\Zz$ is chosen large enough. Furthermore, again by Definition \ref{transformation} (and the translation of $\widehat D_{n_k,1}$ of $q$-multiples) property (ii) is verified. Finally, property (iii) directly follows from 
\eqref{contradiction} and \eqref{eqivan101}. 

If condition \eqref{eqivan101} is not satisfied, the definition of  $\widetilde D_{n_k}^1$ and  $\widetilde D_{n_k}^2$ is more involved. We choose an order among the connected components of $\widehat D_{n_k}$ other than $\widehat D_{n_k,1}$, say $\widehat D_{n_k,\ell}$ for $\ell\geq2$  with the convention that $\widehat D_{n_k,\ell}:=\emptyset$ for $\ell$ larger than the number of connected components of $\widehat D_{n_k}$, and we observe that 
$$
0\leq \lim_{k \to \infty} \mu_{\widehat D_{n_k}} (\widehat D_{n_k,\ell})\leq  \lim_{k \to \infty} \mu_{\widehat D_{n_k}} (\widehat D_{n_k,1})=0, 
$$
so that
\begin{equation}\label{componentszero}
\lim_{k \to \infty} \mu_{\widehat D_{n_k}} (\widehat D_{n_k,\ell})=0
\end{equation} 
for all $\ell\in\Nz$. Furthermore, from \eqref{componentszero} and the fact that 
\begin{equation}\label{totalnk}
\mu_{\widehat D_{n_k}}\left( \bigcup_\ell  \widehat D_{n_k,\ell}\right)=1,
\end{equation} 
it follows that there exist $J_k\geq2$ such that 
\begin{equation}
 \sum_{j=1}^{J_k-1}\mu_{\widehat D_{n_k}} (\widehat D_{n_k,\ell_j})\leq\frac13 \quad\text{and}\quad \sum_{j=1}^{J_k}\mu_{\widehat D_{n_k}} (\widehat D_{n_k,\ell_j})>\frac13.
\end{equation}  
We define
$$
\widetilde  D_{n_k}^1:=\mathcal{T}\left(\bigcup_{j=1}^{J_k}\widehat D_{n_k,\ell_j}\right)\quad\text{and}\quad \widetilde D_{n_k}^2:=\mathcal{T}\left(\widehat{D}_{n_k} \backslash \left(\bigcup_{j=1}^{J_k}\widehat D_{n_k,\ell_j}\right)\right)+t'q\tone
$$ 
for a large $t'\in\Nz$. As in the previous case properties (i) and (ii) directly follow from Definition \ref{transformation} 
and the choice of the $\tone$-translation by a  $q$-multiple $t'\in\Nz$ large enough, where $q$ is defined in \eqref{eSratio}.  
Finally, property (iii)  is also satisfied since 
\begin{align*}
& \limsup_{k \to \infty } \mu_{\widetilde D_{n_k}}(\widetilde D_{n_k}^2)\\
  &\geq  \lim_{k \to \infty } \mu_{\widehat D_{n_k}}\left(\bigcup_{\ell\in\Nz}  \widehat D_{n_k,\ell}\right) -  \limsup_{k \to \infty } \mu_{\widetilde D_{n_k}}\left(\bigcup_{j=1}^{J_k-1}\widehat D_{n_k,\ell_j}\right) - \lim_{k \to \infty} \mu_{\widehat D_{n_k}} (\widehat D_{n_k,J_k})\\
 &=1-  \limsup_{k \to \infty } \mu_{\widetilde D_{n_k}}\left(\bigcup_{j=1}^{J_k-1}\widehat D_{n_k,\ell_j}\right) -0\geq \frac23>0,
 \end{align*}
where we used \eqref{totalnk} and \eqref{componentszero}. Therefore, the claim is verified. 
\EEE

By \FFF such claim and \EEE the same arguments used in Theorem \ref{compactnesstheorem} we deduce that (up to a non-relabeled subsequence) 
$$
\mu_{\widetilde D_{n_k}^j}\rightharpoonup^*\frac{1}{|D^j|}\chi_{D^j}
$$
 in $\mathcal{M}(\Rz^2)$ for $j=1,2$, with $D^j$  disjoint bounded sets of finite perimeter such that
$$
D=D^1\cup D^2.
$$
Therefore, if with $\lambda_j:=|D^j|$, then
\begin{equation} 
	\label{sum_lambdas} 
\lambda_1+\lambda_2=|D| =\frac{\sqrt 3}{2}
		\end{equation} 
with both $\lambda_1>0$  and $\lambda_2>0$, respectively, because of \eqref{contradiction} and the fact that $\FFF\textrm{supp } \mu_{\widetilde D_{n_k}^2}\EEE\subset B(x_2,R_2)$ are the connected components of $\widetilde D_{n_k}$ with largest cardinality  in the case when \eqref{eqivan101} is satisfied and because of \eqref{eqivan103} when \eqref{eqivan101} is not satisfied.  
Finally,  by scaling arguments we conclude 
$$ 
m_0=\PPP\mathcal{E}\EEE(\PPP D\EEE)=\PPP\mathcal{E}\EEE(\PPP D^1\EEE)+\PPP\mathcal{E}\EEE(\PPP D^2\EEE)= \GGG \sqrt{\frac{2}{\sqrt 3}} \UUU \left(\sqrt{\lambda_1} m_0 +\sqrt{\lambda_2} m_0\right)
$$
and hence, \PPP by \eqref{claim_positive} and \eqref{sum_lambdas}  we obtain \EEE
$$\sqrt{\lambda_1} +\sqrt{\lambda_2}=\sqrt{\sqrt{3}/2}=\sqrt{\lambda_1+\lambda_2},$$
\PPP which \EEE implies \UUU $\lambda_1=0$ or $\lambda_2=0$ \PPP that \EEE is a contradiction.

\end{proof}

We are now ready to prove  Theorem \ref{thm:convergence_minimizers}.

	 \begin{proof}[Proof of Theorem \ref{thm:convergence_minimizers}]
\PPP Assertions 1.\ and 2. directly follow from Theorem \ref{Thm:Gamma_convergence} and Corollary \ref{Cor:trasformed_convergence}, respectively. It remains to show Assertion 3. to which the rest of the proof is devoted.\EEE

\PPP
Let $\mu_n\in\mathcal{M}_n$ be  minimizers of $E_n$. By Corollary \ref{Cor:trasformed_convergence}  there exist another sequence of minimizers $\widetilde \mu_n\in\mathcal{M}_n$ of $E_n$, an increasing sequence $n_k$ for $k\in\Nz$, and a measure $\mu\in \mathcal{M}_W$ minimizing $I_\infty$ such that   
\begin{equation} \label{conclusion1}
\widetilde \mu_{n_k}\rightharpoonup^*\mu
\end{equation} 
 in $\mathcal{M}(\Rz^2)$. In particular, from the proof of Corollary \ref{Cor:trasformed_convergence}  we observe that  
$$
\widetilde \mu_n:=\mu_{\mathcal{T}(D_n)}(\cdot+t_{n}\tone)
$$
for  some  integers $t_n\in\Zz$, and for configurations $D_n\in\mathcal{C}_n$ such that $\mu_n:=\mu_{D_n}$, where  $\mathcal{T}(D_n):=\mathcal{T}_2(\mathcal{T}_1(D_n))$ (see Definition \ref{transformation}). 
Furthermore,  by \eqref{radon_functional} and \eqref{converging_energies} the configurations  $D_n$ are minimizers of $V_n$ in $\mathcal{C}_n$ and hence, $\mathcal{T}_1(D_n)=D_n$ and by Theorem \ref{connectness} we have that, up to a non-relabelled subsequence, 
\begin{equation} \label{conclusion2}
\lim_{k \to \infty} \mu_{D_{n_k}} (D_{{n_k},1})=1, 
\end{equation} 
where $D_{{n_k},1}$ is a connected component of $D_{n_k}$ (with the largest cardinality). We also observe that the transformation $\mathcal{T}_2$ consists in translations  of the connected components of $D_{n_k}$ with respect to a vector  in the direction $-\tone$ with norm (depending on the component) in $\Nz\cup\{0\}$. Let  $t'_{n_k}\in\Nz\cup\{0\}$ be the norm of the vector for the connected component $D_{{n_k},1}$.  From \eqref{conclusion1} and \eqref{conclusion2} it follows that
$$
\mu_{D_{n_k}}(\cdot+(t_{n_k}-t'_{n_k})\tone)\rightharpoonup^*\mu 
$$
and hence, we can choose $c_{n_k}:=t_{n_k}-t'_{n_k}\in\Zz$.

\EEE

\end{proof}

\EEE

\section*{\PPP Appendix}

\PPP In this Section we collect some auxiliary results used in the proofs of previous sections for the convenience of the Reader.\EEE
\PPP \begin{lemma}\label{lmabscon}
	 Let $\Omega \subset \FFF \mathbb{R}^d\EEE$ be an open set. Let $\kappa_n$ and $\mu_n$ be two sequences of finite (positive) Borel measures \FFF on the $\sigma$-algebra on $\Omega$ denoted by $\mathcal{B} (\Omega)$ \EEE such that:
	 \begin{itemize}
	 \item[(i)] {$\sup_{n \in \mathbb{N} } (\mu_n(\Omega)+\kappa_n(\Omega))< \infty$},
	 \item[(ii)]  $(\kappa_n)_n$ is uniformily absolutely continuous with respect to $(\mu_n)_n$, i.e., for every $\varepsilon>0$ there exists $\delta>0$ such that  
	$$ \mu_n(A) < \delta \implies \kappa_n (A) < \varepsilon, $$ 
	 for every   $A \in \mathcal{B} (\Omega)$ and $n \in \mathbb{N}$. 
	 \end{itemize}
	 If there exist Borel measures $\kappa$ and $\mu$ on $\mathcal{B} (\Omega)$ such that
	$\kappa_n \stackrel{*}{\rightharpoonup} \kappa$ and $\mu_n \stackrel{*}{\rightharpoonup} \mu$, then $\kappa$ is absolutely continuous with respect to $\mu$.
\end{lemma}\EEE
\begin{proof} 
	Take $A \subset \Omega$ such that $\mu(A)=0$. Since $\kappa$ is regular Borel measure it is enough to prove that $\kappa(K)=0$ for every $K \subset A$, $K$ compact. Take an arbitrary $K \subset A$ compact and $\varepsilon>0$. By regularity of $\mu$ there exists $U \subset X$ open such that $A \subset U$ and $\mu(U) < \delta$, where $\delta$ is given by \PPP (ii). 
\EEE 	For every $x \in K$ we find a ball of radius $r_x$ such that $\mu (\partial B(x,r_x))=0$ and $B(x,r_x) \subset U$. Since $K$ is compact we can find a finite number of balls $(B(x_i, r_{x_i}))_{i=1,\dots, n}$ that cover $K$ and we define an open set $V \subset U$ as $V:= \cup_{i=1}^n B(x_i, r_{x_i})$. Obviously $\mu (\partial V)=0$. We have that $\mu (V) < \delta$ and $\mu_n(V) \to \mu(V)$. Thus there exists $n_0 \in \mathbb{N}$ such that $\mu_n(V) < \delta$, $\forall n \geq n_0$.  But then we have that  $\kappa_n(V)<\varepsilon$, $\forall n \geq n_0$. By the definition of weak star convergence we also have that
	$$ \kappa(K) \leq \kappa(V) \leq \liminf_{n \to \infty} \kappa_n(V) < \varepsilon. $$
	The claim follows by the arbitrariness of $\varepsilon$.   
\end{proof}
We now recall some claims about the \PPP \emph{special functions of bounded variation}  on \FFF an open \EEE set $A\subset\Rz^2$, namely $SBV(A)$  (see \cite{AFP} for more details). \EEE
We recall that the distributional gradient $Dg$ of a function $g \in SBV (A)$ can be decomposed
as:
$$ Dg = \nabla g \mathcal{L}^2 \measurerestr A+(g
^+-g^-) \otimes  \nu_g  \mathcal{H}^1 S_g, $$
where $\nabla g$ is the approximate gradient of $g$ , $S_g$ is the jump set of $g$, $\nu_g$ is a unit normal to
$S_g$, and $g^{\pm}$ are the approximate trace values of $g$ on $S_g$.
We recall a compactness result in $SBV$ and $SBV_{\textrm{loc}}$. 
\begin{theorem} \label{thsbv0}  Let A be bounded and open and let  $g_n$ be a sequence in $SBV (A)$. Assume that
there exists $p>1$ and $C>0$ such that
\begin{equation} \label{sbv1} 
\int_A |\nabla g_n|^p+\mathcal{H}^1(S_{g_n})+\|g_n\|_{L^{\infty} (A)}\leq C, \textrm{ for all } n \in \mathbb{N}. 
\end{equation}
Then, there exists $g \in SBV (A)$ such that, up to a subsequence,
\begin{enumerate} 
	\item
$g_n  \to g$ strongly in $L^1(A)$, 
\item $\nabla g_n \rightharpoonup \nabla g$ \textrm{weakly in } 
$L^1(A, \mathbb{R}^2)$,
\item $\liminf_{n \to \infty} \mathcal{H}^1 (S_{g_n} \cap A') \geq \mathcal{H}^1(S_g \cap A')$, for every open set $A' \subseteq A$.
\end{enumerate} 
\end{theorem}  
 We say that  $g_n \in SBV(A)$  weakly converge in $SBV(A)$ to
a function $g \in SBV(A)$, and we write that $g_n \rightharpoonup g$  in $SBV(A)$, if $g_n$ satisfy \eqref{sbv1}  for some
$p>1$ and $g_n \to g$ in $L^1(A)$. The corollary below easily follows by Theorem \ref{thsbv0}.
\begin{corollary} \label{corsbv2} 
 Let $g_n$ be a sequence in $SBV(\mathbb{R}^2)$. Assume that there exists $p>1$ and $C>0$  such that
\begin{equation} 
\int_{\mathbb{R}^2} |\nabla g_n|^p +\mathcal{H}^1(S_{g_n})+\|g_n\|_{L^{\infty} (\mathbb{R}^2)} \leq C. 
\end{equation} 
Then there exists $g \in SBV(\mathbb{R}^2)$ such that, up to a subsequence,  $(1)-(3)$ of Theorem \ref{thsbv0} holds for every open bounded set $A \subset \mathbb{R}^2$. 
\end{corollary} 
We say that a sequence $g_n$ which belongs to $SBV_{\textrm{loc}}(\mathbb{R}^2)$
weakly converges in $SBV_{\textrm{loc}}(\mathbb{R}^2)$ to a function $g \in SBV_{\textrm{loc}}(\mathbb{R}^2)$, and we write that $g_n \rightharpoonup g$ in $SBV_{\textrm{loc}}(\mathbb{R}^2)$, if $g_n \rightharpoonup g$ in $SBV(A)$ for every
open bounded set $A\subset \mathbb{R}^2$.

 \section*{Acknowledgments}
 \PPP The author are thankful to the Erwin Schr\"odinger Institute in Vienna, where part of this work was developed during the ESI Joint Mathematics-Physics Symposium ``Modeling of crystalline Interfaces and Thin Film Structures'', and acknowledges the support received from BMBWF through the OeAD-WTZ project HR 08/2020.
P. Piovano  acknowledges support from the Austrian Science Fund (FWF) project P 29681 and from the Vienna Science and Technology Fund (WWTF) together with the City of Vienna and Berndorf Privatstiftung through Project MA16-005. I. Vel\v{c}i\'c  acknowledges support from the Croatian Science Foundation under grant no. IP-2018-01-8904 (Homdirestproptcm).\EEE

\end{document}
